%% file: fw_nonlinear_polytopes.tex
\documentclass[11pt]{article}

\usepackage[margin=1.05in]{geometry}
\usepackage{amsmath,amssymb,amsthm,mathtools}
\usepackage[T1]{fontenc}
\usepackage{lmodern}
\usepackage{microtype}
\usepackage{tikz}
\usepackage{float}
\usepackage{subcaption}
\usetikzlibrary{arrows.meta,calc,decorations.pathreplacing}
\usepackage{pgfplots}
\pgfplotsset{compat=1.18}
\usepackage[section]{placeins}
\usepackage{xurl}
\usepackage[colorlinks=true,linkcolor=blue,citecolor=blue,urlcolor=blue]{hyperref}
\makeatletter
\@ifundefined{ver@hdvips.def}{}{\usepackage[hyphenbreaks]{breakurl}}
\makeatother
\hypersetup{
  pdftitle={Last-Iterate Convergence of Frank--Wolfe for Nonlinear Variational Inequalities on Polytopes},
  pdfauthor={Tobias Harks}
}

\newtheorem{theorem}{Theorem}
\newtheorem{lemma}{Lemma}
\newtheorem{proposition}{Proposition}
\newtheorem{corollary}{Corollary}

\theoremstyle{definition}

\theoremstyle{remark}

\newcommand{\R}{\mathbb{R}}
\newcommand{\ip}[2]{\langle #1,#2\rangle}
\newcommand{\norm}[1]{\lVert #1\rVert}

\DeclareMathOperator{\conv}{conv}

\DeclareMathOperator*{\argmin}{arg\,min}

\title{Last-Iterate Convergence of Frank--Wolfe\\
for Nonlinear Variational Inequalities on Polytopes}
\author{Tobias Harks\thanks{Chair of Mathematical Optimization,
University of Passau, Dr.-Hans-Kapfinger-Str.~30, 94032 Passau, Germany.
Email: \href{mailto:tobias.harks@uni-passau.de}{\texttt{tobias.harks@uni-passau.de}}.}}
\date{October 12, 2026}
\begin{document}
\maketitle
\begin{abstract}
We consider the Frank--Wolfe algorithm for solving variational
inequalities over compact polytopes under a monotone $C^1$ operator
$\Phi$. The linear minimization oracle returns vertices, and we assume
a coercivity condition on the Jacobian of $\Phi$ at a solution.
We prove last-iterate convergence to the unique
solution at rate $O(k^{-\min\{\alpha,1\}})$ for step sizes
$\gamma_k=\alpha/(k+\beta)$ with $\beta\geq\alpha>0$. The bound holds
uniformly over starting points and vertex tie-breaking. In particular,
both $1/(k+1)$ and $2/(k+2)$ give $O(k^{-1})$ distance. A matching
lower bound shows that this distance guarantee is worst-case
order-optimal among prescribed admissible step-size schedules.
We also obtain rates $O(k^{-\min\{\alpha,1\}})$ for the Frank--Wolfe gap
and $O(k^{-\min\{\alpha,2\}})$ for the dual gap. These yield
$O(\varepsilon^{-1})$ and $O(\varepsilon^{-1/2})$ calls to each oracle
for sustained gap accuracy $\varepsilon$ when $\alpha\geq1$ and
$\alpha\geq2$, respectively, with instance-dependent constants.
We also treat more general prescribed step-size schedules.
\end{abstract}

\section{Introduction}\label{sec:introduction}
Let $K$ be a compact polytope in a finite-dimensional Euclidean space $E$
and let $\Phi\colon K\to E$ be a monotone mapping. We study the simplest
projection-free method for the variational inequality
\begin{equation}\label{eq:vi}
 \text{Find }x^*\in K\text{ such that }\ip{\Phi(x^*)}{s-x^*}\geq0\qquad\forall s\in K,
\end{equation}
namely the Frank--Wolfe (FW) or conditional gradient iteration with a
vertex-returning linear minimization oracle and prescribed step sizes,
\begin{equation}\label{eq:FW-intro}
 \begin{split}
 x_{k+1}&=x_k+\gamma_k(s_k-x_k),\qquad \gamma_k=\frac{\alpha}{k+\beta},\\
 s_k&\in\operatorname{vert}(K)\cap\argmin_{s\in K}\ip{\Phi(x_k)}s,
       \qquad k\geq0.
 \end{split}
\end{equation}
Each update evaluates $\Phi$ once, minimizes a linear function over $K$
once, and moves toward the returned vertex. This is useful when linear
minimization is substantially cheaper than projection onto $K$. The
iterate stays feasible, and after $k$ updates from a vertex it is a
convex combination of at most $k+1$ vertices, see Jaggi~\cite{Jaggi2013} and Pokutta~\cite{Pokutta2024}. When $\Phi=\nabla f$ for a convex $f$,
\eqref{eq:FW-intro} is the classical Frank--Wolfe method with open-loop
steps.

When $K$ is a product
of simplices and $\Phi$ is the payoff mapping of a bilinear zero-sum
game, \eqref{eq:FW-intro} with $\gamma_k=1/(k+1)$ corresponds to  simultaneous
\emph{fictitious play}, where each player best-responds to the empirical average of
past play, and $x_k$ is that average.
This dynamic has received substantial interest in game theory, starting with the early works by~Brown~\cite{Brown51} and Robinson~\cite{robinson1951}.
Hammond \cite{Hammond1985} introduced \emph{generalized fictitious play} for general 
variational inequalities and conjectured convergence of this harmonic
iteration for strongly monotone mappings on polytopes. In monotone
population games, Hadikhanloo et al.\ \cite{Hadikhanloo2021} established
qualitative convergence of the harmonic empirical-state iteration.
Very recently, Hough
\cite{Hough2026} proved the conjecture and showed that for monotone
$C^1$ mappings on compact convex sets and vanishing, nonsummable steps,
the Frank--Wolfe gap and the distance to the solution set
tend to zero. Under strong monotonicity the iterates converge to the
unique solution. Hough also proves quantitative rates for strongly convex
feasible sets. For general polytopes, his result is qualitative and
quantitative rates remain open in his analysis.

This paper proves last-iterate distance rates for the direct vertex-returning
FW iteration with nonlinear monotone $C^1$ operators on compact polytopes.
We assume coercivity of the Jacobian only in directions parallel to the
face exposed by the operator at the solution. This condition is strictly
weaker than global strong monotonicity and permits boundary solutions.
For prescribed steps $\gamma_k=\alpha/(k+\beta)$,
$\beta\geq\alpha>0$, we prove
$\norm{x_k-x^*}=O(k^{-\min\{\alpha,1\}})$, uniformly over starting points
and vertex tie-breaking. In particular, both $1/(k+1)$ and $2/(k+2)$
give $O(k^{-1})$ distance. The guarantee is tight among prescribed
positive, vanishing, nonsummable schedules: for every such schedule,
there exists a two-dimensional quadratic problem admitting a trajectory that
is not $o(k^{-1})$. More precisely, the lower bound combines oscillation
on the scale of the current step with the weight retained after the last
unit step. For eventually monotone rescaled weights, it matches the
upper bound up to constants over two consecutive iterates (see Section~\ref{sec:optimality}). The distance estimate also
yields rates for both VI gaps and corresponding oracle-call bounds.

Two features make these quantitative bounds difficult. On a polytope,
a vertex-returning oracle can switch discontinuously as the score
$\Phi(x_k)$ changes, so the oracle regularity used for rates on strongly
convex sets is unavailable. Moreover, $\Phi$ need not be a gradient,
and objective descent is then unavailable. Monotonicity alone does not
give the claimed rate: fictitious play can have arbitrarily slow
polynomial gap decay, even with unique best responses after
initialization \cite{Wang2025,AbernethyLazarsfeldWibisono2026}.
Our additional coercivity assumption supplies the local control used to
obtain the distance rate. The proof combines qualitative attraction with
an exact rescaling to a weighted additive update dynamic. Such dynamics have been recently
analyzed in Harks~\cite[Theorem~1 and Corollary~1]{RobustCycling}, where the
theorem bounds the state norms generated by these updates; the recursion and
this boundedness result are explained in Section~\ref{subsec:proof-technique}.

\subsection{Our Contribution}\label{subsec:contribution}
The following theorem is our main result on last-iterate convergence.
\begin{theorem}[Uniform last-iterate rate]\label{thm:app-fw}
Suppose $\Phi$ is monotone on $K$ and $C^1$ near $K$, and its Jacobian
at a VI solution $x^*$ is coercive in directions parallel to the exposed
face $F_*:=\argmin_{s\in K}\ip{\Phi(x^*)}s$, in the sense of
\eqref{eq:face-coercivity}. Then $x^*$ is unique and there
is a finite constant
$C_{\rm FW}=C_{\rm FW}(K,\Phi,\alpha,\beta)$ such that every trajectory of
\eqref{eq:FW-intro} satisfies
\[
 \norm{x_T-x^*}\leq C_{\rm FW}T^{-\min\{\alpha,1\}}\qquad(T\geq1).
\]
The constant is independent of the starting point and all vertex tie-breaking
choices. In
particular, every fixed $\alpha\geq1$ gives an $O(T^{-1})$ distance rate.
\end{theorem}

The exponent is determined solely by the step-size parameter; dependence
on the problem instance enters through the constant.

For a $C^1$ globally strongly monotone mapping on any compact polytope,
Theorem~\ref{thm:app-fw} gives the distance guarantee $O(k^{-1})$ for
$1/(k+1)$ and $2/(k+2)$; this is the quantitative form of Hammond's
conjecture. The mapping may be nonlinear, its derivative nonsymmetric,
and the solution degenerate on the boundary. 
As an application, consider atomic splittable network congestion games. Their strategy sets are flow polytopes, and under the operator conditions in~\cite{harks2026} the marginal-cost mapping satisfies the hypotheses of Theorem~\ref{thm:app-fw}, so the harmonic vertex method converges at rate $O(k^{-1})$.

\paragraph{Matching Lower Bounds.}
The following result shows that the distance rate cannot be improved
by choosing a different prescribed schedule. Its lower bound combines
two mechanisms on one strongly convex quadratic problem: vertex
selections force oscillation on the scale of the current step, while a
boundary coordinate retains a positive fraction of an earlier error.
We start the comparison at an index $n\geq1$ after all unit steps, so
this retained fraction is the positive product
$\prod_{j=n}^{k-1}(1-\gamma_j)$. Such an index exists because
$\gamma_k\to0$. Changing $n$ changes constants, but not asymptotic orders.

\begin{theorem}[Matching lower bounds]
\label{thm:schedule-characterization}
Fix a deterministic schedule with $0<\gamma_k\leq1$, $\gamma_k\to0$,
and $\sum_k\gamma_k=\infty$. Choose $n\geq1$ such that $\gamma_k<1$
for all $k\geq n$, and define
\begin{equation}\label{eq:schedule-normalization}
 \rho_n=1,\qquad \rho_{k+1}=\frac{\rho_k}{1-\gamma_k},\qquad
 w_k=\rho_{k+1}\gamma_k,\qquad
 q_k=\max\{\gamma_k,\rho_k^{-1}\}.
\end{equation}
On the fixed instance $K=[0,1]^2$, $\Phi(\xi,\eta)=(\xi-1/2,\eta)$,
there are a starting point, exact vertex selections, and a constant
$c>0$, depending only on the schedule and $n$, such that
\[
 \Delta_k:=\max\{\norm{x_k-x^*},\norm{x_{k+1}-x^*}\}
 \geq c q_k\qquad\text{for all sufficiently large }k.
\]
This trajectory is not $o(k^{-1})$. Consequently, the $O(k^{-1})$
distance guarantee in Theorem~\ref{thm:app-fw} cannot be improved to
$o(k^{-1})$ uniformly over problems by any such prescribed schedule.
\end{theorem}

Here $\rho_k^{-1}=\prod_{j=n}^{k-1}(1-\gamma_j)$ measures how much of
the state at time $n$ remains at time $k$. Thus $q_k$ compares the size
of a new step with the retained weight of an earlier state. The
maximum over two consecutive iterates captures oscillation: an iterate
may be unusually close to the solution even when the next one is not.

For the reciprocal schedules of Theorem~\ref{thm:app-fw}, this lower
bound matches that theorem's rate. More generally, smaller steps reduce
vertex oscillation but preserve the influence of earlier iterates for
longer. The upper and lower bounds agree up to constant factors whenever
the rescaled weights are eventually monotone. The precise bounds for
general schedules and their comparison with the lower bound are given
in Corollary~\ref{cor:general-schedules} and
Section~\ref{sec:optimality}.

\paragraph{Gap accuracy and oracle calls.}
We further investigate the Frank--Wolfe gap and the dual (Minty) gap
along FW-trajectories:
\begin{equation}\label{eq:vi-gaps}
 V(x):=\max_{s\in K}\ip{\Phi(x)}{x-s},\qquad
 G(x):=\max_{s\in K}\ip{\Phi(s)}{x-s}.
\end{equation}
Monotonicity gives $0\leq G\leq V$, and both vanish exactly at solutions.
The distance bound controls $V$ linearly. For $G$, face coercivity yields
quadratic growth of $s\mapsto\ip{\Phi(s)}{s-x^*}$, and the linear term
$\ip{\Phi(x^*)}{x_k-x^*}$ decays exactly at the retained-weight rate
$k^{-\alpha}$ once the selected vertices lie in $F_*$. For steps
$\alpha/(k+\beta)$ this gives (Corollary~\ref{cor:gap-distance})
\[
 V(x_k)=O(k^{-\min\{\alpha,1\}}),\qquad
 G(x_k)=O(k^{-\min\{\alpha,2\}}).
\]
Since each update costs one evaluation of $\Phi$ and one vertex-LMO call,
Table~\ref{tab:accuracy-complexity} translates the best exponents into
oracle-call bounds for \emph{sustained accuracy}, meaning that the
tolerance holds for every subsequent iterate. Both gap exponents are
tight over two consecutive iterates (Proposition~\ref{prop:gap-sharpness}),
and the sustained-accuracy orders are optimal among prescribed schedules
(Corollary~\ref{cor:oracle-tail}). Notably, $\alpha\geq2$ improves the
worst-case dual-gap guarantee over harmonic steps, although the distance
order is the same.

\begin{table}[H]
\centering
\small
\renewcommand{\arraystretch}{1.25}
\begin{tabular}{@{}lccc@{}}
\hline
Accuracy criterion & Step condition & Rate & \shortstack{Calls to each\\oracle}\\
\hline
$\norm{x_k-x^*}\leq\varepsilon$ & $\alpha\geq1$ & $O(k^{-1})$ & $O(\varepsilon^{-1})$\\
$V(x_k)\leq\varepsilon$ & $\alpha\geq1$ & $O(k^{-1})$ & $O(\varepsilon^{-1})$\\
$G(x_k)\leq\varepsilon$ & $\alpha\geq2$ & $O(k^{-2})$ & $O(\varepsilon^{-1/2})$\\
\hline
\end{tabular}
\caption{Accuracy and oracle-call guarantees for
$\gamma_k=\alpha/(k+\beta)$, $\beta\geq\alpha$, on admissible instances.
Calls are counted separately for mapping evaluations and
vertex-returning linear minimization. The constants depend on the
instance and the schedule and are uniform over initial points and vertex
choices.}
\label{tab:accuracy-complexity}
\end{table}

The rate exponents hold across all admissible instances, while the
implicit constants may depend on $K$, $\Phi$, and the schedule.
These bounds do not provide class-uniform minimax complexity guarantees;
Section~\ref{sec:oracle-counts} explains this distinction using the
standard simplex obstruction.
Also, $V(x_k)$ is an observable stopping certificate obtained from the
same oracle pair, whereas the bound on $G(x_k)$ is a guarantee on the
returned iterate and not a procedure for evaluating $G$.

\subsection{Proof Technique}\label{subsec:proof-technique}
The proof has three stages. The first follows Hough's interpolation
and limit-set approach \cite[Section~2]{Hough2026}, within the
established differential-inclusion framework of Bena\"im, Hofbauer, and
Sorin \cite{BenaimHofbauerSorin2005}; related interpolation arguments for
fictitious play appear in \cite[Section~5.1]{Hadikhanloo2021}.
We adapt it to obtain attraction
to $x^*$ uniformly over all starting points and vertex tie-breaking
choices. Time-shifted
interpolations of the iterates have limit trajectories of the
associated differential inclusion, and exponential decay of the
Frank--Wolfe gap along these limits yields a common entry time into
every neighborhood of the solution. Since the vertex set is finite, all
selected vertices then belong to the exposed face $F_*$ from a common
index onward. The iterates themselves need not enter that face in finite
time.

The second stage is algebraic. An exact rescaling turns the averaging
recursion into an additive one: with $\rho_{k+1}(1-\gamma_k)=\rho_k$ and
$w_k=\rho_{k+1}\gamma_k$, the state $z_k=\rho_k(x_k-x^*)$ satisfies
\[
 z_{k+1}=z_k+w_k(s_k-x^*).
\]
After vertex-face identification, the component of $z_k$ normal to the
face's direction space $T_*:=\operatorname{lin}(F_*-F_*)$ is constant.
The linearized oracle scores depend on $A_*y_k+c$, where $y_k\in T_*$
is the tangential component, $A_*$ is the Jacobian restricted to $T_*$
and projected onto $T_*$, and $c$ is constant. Coercivity makes $A_*$
invertible. The substitution $\widetilde y_k=y_k+A_*^{-1}c$ gives
$A_*\widetilde y_k=A_*y_k+c$ and preserves the additive update.
With $U_*:=\operatorname{vert}(F_*)-x^*$, the reduced update has the form
\[
 \widetilde y_{k+1}=\widetilde y_k+w_ku_k,\qquad
 u_k\in U_*,\qquad
 \ip{A_*\widetilde y_k}{u_k}
 \leq\min_{u\in U_*}\ip{A_*\widetilde y_k}u+b_k,
\]
where $b_k\geq0$ accounts for the nonlinear Taylor remainder.
We call this a \emph{weighted approximate perceptron recursion}: each
update adds a direction whose linear score is within $b_k$ of the
minimum. For affine mappings, $b_k=0$.

The third stage applies the robust weighted perceptron cycling theorem
of Harks~\cite[Theorem~1 and Corollary~1]{RobustCycling}, stated as
Theorem~\ref{thm:weighted} below. The name refers to a boundedness
result for these additive updates; periodic repetition is not required.
For a fixed starting index $N$ and any later index $m$, it gives
\[
 \max_{N\leq k\leq m}\norm{\widetilde y_k}
 \leq C_w\left(\norm{\widetilde y_N}
       +\max_{N\leq k<m}w_k+\max_{N\leq k<m}b_k\right).
\]
The constant $C_w$ is independent of $m$ and the tolerances. Near the
solution, the Taylor contribution to $b_k$ is bounded by an arbitrarily
small multiple of $\norm{\widetilde y_k}$ plus a fixed constant.
Choosing the neighborhood small enough lets us move this contribution
to the left-hand side. The resulting bound holds with the same constant
for every $m>N$, starting from the same index $N$. No Lipschitz
assumption on $D\Phi$ or prior rate of attraction is needed.

Figure~\ref{fig:proof-roadmap} illustrates the geometric reduction on an
affine example. In subfigure~\subref{subfig:face-entry}, the selected
vertices lie in $F_*$ while the iterates approach it. In
subfigure~\subref{subfig:tangential-cycling}, the rescaled and shifted
state adds either $-1/2$ or $1/2$ and eventually alternates between
$0$ and $1/2$. This is a periodic example of the additive recursion;
the theorem controls general trajectories without assuming such
periodicity. Nonlinear mappings introduce the tolerance $b_k$ described
above.

\input{proof_roadmap.tex}

The dual-gap bound needs one further ingredient: face coercivity also
yields quadratic growth of $s\mapsto\ip{\Phi(s)}{s-x^*}$, giving
\[
 G(x)\leq\ip{\Phi(x^*)}{x-x^*}+C\norm{x-x^*}^2.
\]
After face identification, the first term decays exactly at the
retained-weight rate $k^{-\alpha}$. Combining it with squared distance
gives the exponent $\min\{\alpha,2\}$ and explains the
$O(\varepsilon^{-1/2})$ oracle-call bound for $\alpha\geq2$.

\subsection{Organization}
Section~\ref{sec:related} reviews related work.
Section~\ref{sec:analysis} proves the last-iterate bound, and
Section~\ref{sec:consequences} develops general schedules, gap bounds,
matching lower bounds, and their oracle-count interpretation.
Section~\ref{sec:outlook} discusses extensions to inexact oracles and
bounded delays, limitations, and open problems, including possible
extensions beyond polytopes.
Appendix~\ref{sec:compactness} contains the differential-inclusion details,
and Appendix~\ref{app:extension-proofs} gives the proofs of the extensions.

\section{Related Work}\label{sec:related}
\paragraph{Variational inequalities with LMOs.}
Juditsky and Nemirovski \cite{JuditskyNemirovski2016} consider  linear minimization oracles over $K$.
Their algorithm uses a representation of $\Phi$ to construct an
auxiliary, dual VI on a domain where proximal steps are tractable.
They solve that VI with Mirror Descent or Mirror-Prox, calling the
linear oracle over $K$ to obtain primal points. Their Theorem~1
recovers a weighted combination of those points with a certified bound
on the dual VI gap $G$ in \eqref{eq:vi-gaps}.
Thus both methods return combinations of linear-oracle outputs, but
choose them differently: their query directions and averaging weights
come from the auxiliary algorithm, whereas \eqref{eq:FW-intro} queries
$\Phi(x_k)$ and uses the prescribed step $\gamma_k$. Their framework
provides quantitative VI solvers without our face-coercivity assumption;
our result gives distance and gap rates for the direct FW iteration
under the additional polytopal and coercivity hypotheses.

Other approaches combine proximal and linear-oracle steps
\cite{HeHarchaoui2015,ChenEtAl2020}, study FW for saddle-point problems
\cite{GidelEtAl2017}, or use inner FW loops to solve minimax
subproblems \cite{RahimiEtAl2026}. In methods with inner loops, an outer
iteration can require many LMO calls. Chen and Mazumdar
\cite{ChenMazumdar2024} obtain last-iterate rates for a generalized FW
method with a smoothed oracle.

\paragraph{Mirror methods and last-iterate rates.}
Mertikopoulos and Staudigl \cite{MertikopoulosStaudigl2018VI} establish
gap rates for averaged trajectories of stochastic mirror-descent dynamics
in monotone VIs. Azizian, Iutzeler, Malick, and Mertikopoulos analyze local
last-iterate rates for stochastic optimistic mirror descent
\cite{AzizianEtAl2021} and deterministic Bregman proximal methods
\cite{AzizianEtAl2024}. They relate the rates to the geometry of the
regularizer near the solution and, in the latter work, to sharp directions
at boundary solutions. Their rate analyses use proximal updates and
Bregman geometry. We obtain Euclidean distance and VI-gap bounds
for vertex-returning FW under coercivity along the exposed face.

\paragraph{Fictitious play and best-response dynamics.}
For an $m\times n$ zero-sum game, Shapiro \cite{Shapiro1958} gives an
$O(k^{-1/(m+n-2)})$ gap bound. Daskalakis and Pan
\cite{DaskalakisPan2014} obtain $\Omega(k^{-1/n})$ on the $n\times n$
identity game with adversarial tie-breaking, while Abernethy, Lai, and
Wibisono \cite{AbernethyLaiWibisono2021} prove $O(k^{-1/2})$ for positive
diagonal payoffs with lexicographic tie-breaking. Wang \cite{Wang2025}
gives an $\Omega(k^{-1/3})$ example without ties after initialization;
Abernethy, Lazarsfeld, and Wibisono
\cite{AbernethyLazarsfeldWibisono2026} extend this to arbitrarily slow
polynomial gap rates. These examples have skew game operators, for
which $\ip{D\Phi h}h=0$, and hence fail our coercivity assumption when
$T_*\ne\{0\}$. Their saddle-point gap equals both $V$ and $G$; it is
not a distance estimate.

The continuous-time analysis draws on best-response dynamics
\cite{HofbauerSorin2006}, the gap inequality $\dot V\leq-V$ for
monotone VIs \cite{ChenKA24}, and the differential-inclusion
framework of Bena\"im, Hofbauer, and Sorin
\cite{BenaimHofbauerSorin2005}. Hadikhanloo, Laraki, Mertikopoulos, and
Sorin \cite[Theorems~4.1 and~5.1]{Hadikhanloo2021} prove qualitative
convergence of best-response dynamics and fictitious play in monotone
population games. Their empirical state is the harmonic FW iterate on
a simplex, and their discrete-time proof uses the same differential-inclusion
framework. Hough \cite[Lemma~5 and Theorem~6]{Hough2026}
uses interpolation and complete trajectories in the limit set to prove
convergence of direct FW. Section~\ref{sec:attraction} adapts this
established method to obtain a common attraction time across starting
points and exact oracle selections.

\paragraph{Herding and perceptron cycling.}
Gelfand et al.\ \cite{GelfandEtAl2010} use classical perceptron cycling
to prove $O(k^{-1})$ moment error for finite herding, including boundary
targets and nongreedy selections satisfying a sign condition. Bach,
Lacoste-Julien, and Obozinski \cite{BachEtAl2012} establish the
herding--conditional-gradient correspondence. The harmonic case
$\Phi(x)=x-\mu$, $\mu\in K$, therefore has a direct precedent.
Classical effective bounds and the geometry of herding constants are
studied in \cite{AmaldiHauser2005,HarveySamadi2014}.
Harks \cite{RobustCycling} proves the robust weighted cycling estimate
used here and derives affine FW bounds for relative-interior solutions.
Our reduction extends its application to nonlinear mappings and boundary
solutions through uniform attraction, a fixed change of tangential
coordinates, and a bound on the nonlinear Taylor remainder.

\paragraph{Rates in convex optimization.}
Wolfe \cite{Wolfe1970}, Gu\'elat and Marcotte
\cite{GuelatMarcotte1986}, and Garber \cite{Garber2020} analyze away-step
FW under strict complementarity; Garber's rate reflects the optimal
face's dimension. Our face identification concerns the selected
vertices and allows $x^*$ on the relative boundary of $F_*$.
For prescribed steps, Bach \cite[Proposition~2.2]{Bach2021} proves
$O(k^{-2})$ objective error for the unextrapolated update with step
$2/(k+1)$. His assumptions include three derivatives with bounded third
derivative, a positive-definite Hessian at the unique minimizer, and
relative interiority in a simplicial face exposed exactly by the optimal
gradient. Wirth, Kerdreux, and Pokutta \cite{WirthKerdreuxPokutta2023}
and Wirth, Pe\~na, and Pokutta
\cite{WirthPenaPokutta2025,WirthPenaPokutta2026,WirthPenaPokuttaPolytopes2026}
develop accelerated objective rates using growth conditions and
polytope geometry. These results explain why faster FW rates are
possible for potential mappings; our distance and VI-gap bounds permit
nonlinear operators without a potential.

\paragraph{Schedule lower bounds and oracle complexity.}
Pokutta's objective-error lower bound
\cite[Proposition~4.6]{Pokutta2024} contains the retained-weight product
$\prod_j(1-\gamma_j)$. Its coefficient vanishes on our quadratic example;
our distance lower bound combines retained weight with vertex
oscillation. Lan \cite{Lan2013} studies oracle lower bounds over classes
of problems. Section~\ref{sec:oracle-counts} explains how that complexity
notion differs from our tight bounds for prescribed schedules, whose
constants may depend on the instance.

\section{Proof of the Last-Iterate Bound}\label{sec:analysis}
Throughout, $\Phi$ is $C^1$ on an open neighborhood of $K$ and monotone
on $K$:
\[
 \ip{\Phi(x)-\Phi(y)}{x-y}\geq0\qquad\forall x,y\in K.
\]
Continuity and compact convexity give existence of a VI solution. Let
$x^*$ be a solution and define the face exposed by $\Phi(x^*)$ and its
direction space,
\[
 F_*:=\argmin_{s\in K}\ip{\Phi(x^*)}s,\qquad
 T_*:=\operatorname{lin}(F_*-F_*).
\]
We assume that there is $\mu_*>0$ with
\begin{equation}\label{eq:face-coercivity}
 \ip{D\Phi(x^*)h}h\geq\mu_*\norm h^2\qquad\forall h\in T_*\quad\text{(coercivity)}.
\end{equation}
We call this \emph{coercivity on the exposed solution face}.

\begin{lemma}[Uniqueness]\label{lem:unique}
A solution $x^*$ satisfying \eqref{eq:face-coercivity} is the unique VI
solution whenever $\Phi$ is monotone on $K$ and differentiable at $x^*$.
\end{lemma}
\begin{proof}
Suppose $y$ is another solution and put $h=y-x^*$. The two VI inequalities
and monotonicity imply
\[
 0\leq\ip{\Phi(x^*)}h\leq\ip{\Phi(y)}h\leq0.
\]
Thus $y\in F_*$ and $h\in T_*$. Monotonicity along the segment
$x^*+th$, $0\leq t\leq1$, makes
$t\mapsto\ip{\Phi(x^*+th)}h$ nondecreasing. Its endpoint values are both
zero, so it is identically zero. Differentiating at $t=0$ gives
$0=\ip{D\Phi(x^*)h}h\geq\mu_*\norm h^2$, hence $y=x^*$.
\end{proof}

Global strong monotonicity implies \eqref{eq:face-coercivity}: differentiate
the strong-monotonicity inequality at relative-interior points of $K$ along
$T_*\subseteq\operatorname{lin}(K-K)$ and pass to $x^*$ by continuity of
$D\Phi$. The converse fails. For example, on $K=[-1,1]\times[0,1]$, the monotone map
$\Phi(x,y)=(x,1)$ has solution $(0,0)$ and exposed face
$[-1,1]\times\{0\}$. Its derivative is coercive on this face's direction
space, but vanishes in the vertical direction, so the map is not strongly
monotone on $K$.

In \eqref{eq:FW-intro} we fix $\beta\geq\alpha>0$, so that
$0<\gamma_k\leq1$. The starting point $x_0$ is any point of $K$. When
$\beta=\alpha$, the first step has size one and $x_1=s_0$; otherwise the
initial point retains positive weight, and our bounds are uniform over
all $x_0\in K$. The harmonic case $\alpha=\beta=1$ gives
$x_T=T^{-1}\sum_{k=0}^{T-1}s_k$, the plain average of the selected
vertices, while the standard schedule $2/(k+2)$ weights later selections
more heavily. In every case the estimate concerns the current iterate
$x_T$ itself.

For positive sequences $(a_k)$ and $(b_k)$, the notation $a_k\asymp b_k$
means that there are constants $0<c\leq C<\infty$ such that
$c\,b_k\leq a_k\leq C\,b_k$ for all sufficiently large $k$.
We write $a_k\sim b_k$ when $a_k/b_k\to1$.
The notation $a_k=O(b_k)$ means $|a_k|\leq Cb_k$ eventually for some finite
$C$, and $a_k=o(b_k)$ means $|a_k|/b_k\to0$; for vector-valued sequences,
replace absolute value by the Euclidean norm. These constants are independent
of $k$ but may depend on the problem instance and schedule, with further
uniformity specified in each result. The same conventions apply to functions
as their argument tends to the limit under consideration.

We follow the three stages of
Section~\ref{subsec:proof-technique}: uniform attraction and vertex-face
identification in Section~\ref{sec:attraction}, rescaling in
Section~\ref{sec:rescale}, and the weighted perceptron bound in
Section~\ref{sec:cycling}. Section~\ref{sec:rate-proofs} combines them
and bounds the effect of the nonlinear remainder to prove the rates.

\subsection{Uniform attraction and exposed-face identification}\label{sec:attraction}
This is the first stage of the proof. Hough's convergence theorem
\cite[Theorem~6 and Corollary~7]{Hough2026} implies attraction to the
solution set for each fixed trajectory. The following lemma requires
more: for each neighborhood of the solution set, one entry index must
work for every starting point and every sequence of oracle selections.
Convergence of each trajectory separately does not supply this common
index. We need it to make the constant in Theorem~\ref{thm:app-fw}
independent of the trajectory.

We obtain this uniform version by adapting Hough's interpolation and
limit-trajectory argument \cite[Lemma~5 and Theorem~6]{Hough2026}.
In a contradiction sequence, both the trajectory and the observation
time may vary. Compactness still gives a complete limit trajectory,
and Hough's exponential gap decay rules out a limit point away from the
solution set. Thus the additional argument establishes uniformity within
the same proof framework. On a polytope with a unique solution, the
positive score separation from vertices outside $F_*$ then yields a
common index for vertex-face identification.

For $p\in E$, write
\[
 \mathcal B(p):=\argmin_{s\in K}\ip ps.
\]
Fix a deterministic step sequence satisfying
$0<\gamma_k\leq1$, $\gamma_k\to0$, and $\sum_k\gamma_k=\infty$.
Let $\mathfrak X$ consist of all trajectories with $x_0\in K$ and
\[
 x_{k+1}=x_k+\gamma_k(s_k-x_k),\qquad
 s_k\in\mathcal B(\Phi(x_k)).
\]
Here selections may be arbitrary points of $K$; the vertex restriction
is needed only for identification and the rate.

\begin{lemma}[Uniform attraction and  vertex-face identification]
\label{lem:uniform-fw}
Let $K$ be nonempty, compact, and convex, and let $\Phi$ be $C^1$ on a
neighborhood of $K$ and monotone on $K$. If $\mathcal S$ is the VI solution
set, then
\[
 \sup_{(x_k)\in\mathfrak X}\operatorname{dist}(x_k,\mathcal S)\longrightarrow0.
\]
If $\mathcal S=\{x^*\}$, this is uniform convergence to $x^*$. If also $K$
is a polytope, there is a common $k_0=k_0(K,\Phi,(\gamma_k))$ such that all
vertex-returning trajectories in $\mathfrak X$ satisfy
\[
 s_k\in\operatorname{vert}(F_*)\qquad(k\geq k_0),\qquad
 F_*:=\argmin_{s\in K}\ip{\Phi(x^*)}s.
\]
\end{lemma}

\begin{proof}
The proof has three steps: upper semicontinuity of the inclusion,
a backward-complete limit trajectory argument for uniform attraction,
and exclusion of vertices outside the exposed face.

Define the Frank--Wolfe gap and the differential-inclusion map by
\[
 V(x):=\max_{s\in K}\ip{\Phi(x)}{x-s},
 \qquad
 \mathcal G(x):=\mathcal B(\Phi(x))-x.
\]
The map $\mathcal G$ is upper semicontinuous with nonempty compact convex
values.  Indeed, if $x_j\to x$, $g_j\to g$, and
$g_j=s_j-x_j\in\mathcal G(x_j)$, compactness permits passage to
$s_j\to s\in K$, and the inequalities
$\ip{\Phi(x_j)}{s_j}\leq\ip{\Phi(x_j)}{v}$ pass to the limit for every
$v\in K$.  Thus $s\in\mathcal B(\Phi(x))$ and $g=s-x\in\mathcal G(x)$.

Hough proves that $V$ is continuous, that $V^{-1}(0)=\mathcal S$, and that
every solution $x(\cdot)$ of
\[
 \dot x(t)\in\mathcal G(x(t))
\]
satisfies
\begin{equation}\label{eq:hough-gap-decay}
 V(x(t))\leq e^{-(t-r)}V(x(r))
 \qquad(r\leq t)
\end{equation}
\cite[Lemma~2(a)--(b)]{Hough2026}. More precisely, part~(b) is stated
for $K$-valued solutions of $\dot x\in\mathcal B(\Phi(x))-x$, so no projected
extension is needed here. Applying it to the time-shifted solution
$\tau\mapsto x(r+\tau)$ gives the displayed estimate for arbitrary $r\leq t$.

A proof of the gap-decay estimate \eqref{eq:hough-gap-decay} is given in
Appendix~\ref{sec:gap-decay}.

Suppose that the asserted uniform convergence fails.  Then there are
trajectories $(x_k^j)\in\mathfrak X$, integers $k_j\to\infty$, and
$\delta>0$ such that
\[
 \operatorname{dist}(x_{k_j}^j,\mathcal S)\geq\delta.
\]
Set $\sigma_0:=0$ and $\sigma_k:=\sum_{i=0}^{k-1}\gamma_i$, and let $w_j$ be the
piecewise affine interpolation satisfying $w_j(\sigma_k)=x_k^j$.  Every
$w_j$ takes values in $K$ and is $\operatorname{diam}(K)$-Lipschitz.  Shift
the interpolation by defining
\[
 W_j(t):=w_j(t+\sigma_{k_j}),
 \qquad t\in[-\sigma_{k_j},\infty).
\]
On each compact interval, the paths $W_j$ are uniformly bounded and
equi-Lipschitz. Let $\widehat W_j(t)$ denote the left endpoint iterate and
$\ell_j(t)$ its iteration index. For almost every $t$,
\[
 \dot W_j(t)\in\mathcal G(\widehat W_j(t)).
\]
On every fixed compact interval $I$, the smallest index $\ell_j(t)$ tends
to infinity: for any fixed $m$, $\sigma_{k_j}+\inf I>\sigma_m$ eventually.
Since $\gamma_k\to0$, the interpolation error vanishes uniformly:
\[
 \sup_{t\in I}\norm{W_j(t)-\widehat W_j(t)}\to0.
\]
Theorem~\ref{thm:appendix-di-compactness}, applied on $K$ to
$f_j=W_j$, $\widehat f_j=\widehat W_j$, and $G=\mathcal G$, therefore
gives a subsequence converging locally uniformly on $\mathbb R$ to
a complete trajectory $W$ with
$\dot W\in\mathcal B(\Phi(W))-W=\mathcal G(W)$ almost everywhere.

Let $M:=\max_{x\in K}V(x)$.  Applying \eqref{eq:hough-gap-decay} on
$[-T,0]$ gives
\[
 V(W(0))\leq e^{-T}V(W(-T))\leq e^{-T}M
 \qquad(T>0).
\]
Hence $V(W(0))=0$ and $W(0)\in\mathcal S$.  On the other hand,
$W_j(0)=x_{k_j}^j$ and local uniform convergence imply
$\operatorname{dist}(W(0),\mathcal S)\geq\delta$, a contradiction.  This
proves the first assertion. If the solution is unique, distance to
$\mathcal S$ equals distance to $x^*$, giving the second.

It remains to prove identification when the solution is unique and $K$ is
a polytope. If $F_*=K$, the assertion is automatic. Otherwise put
\[
 g_*:=\min_{v\in\operatorname{vert}(K)\setminus F_*}
       \ip{\Phi(x^*)}{v-x^*}>0.
\]
For $x$ close enough to $x^*$, continuity gives
\[
 \ip{\Phi(x)}{v-x^*}\geq g_*/2
 \qquad\forall v\in\operatorname{vert}(K)\setminus F_*.
\]
Choose a common $k_0$ after which every trajectory is in that neighborhood.
A selected vertex outside $F_*$ would satisfy
\[
 0<g_*/2\leq\ip{\Phi(x_k)}{s_k-x^*}\leq0,
\]
since $x^*$ is a feasible comparator in the oracle inequality. This is
impossible. Therefore every selected vertex lies in $F_*$ after $k_0$.
\end{proof}

Lemma~\ref{lem:uniform-fw} supplies a common, but nonexplicit, entry time
into every neighborhood of the unique solution and a common index for
vertex-face identification. Hough explains that the exponential decay
\eqref{eq:hough-gap-decay} for the continuous-time inclusion does not itself
give a discrete rate on a polytope because the LMO is discontinuous
\cite[Section~3.1]{Hough2026}.  We use the uniform attraction lemma only to
obtain uniform vertex-face identification and a common neighborhood where
the Taylor remainder is small. The robust cycling theorem then controls
the discontinuous vertex choices on that face and supplies the discrete rate.

\subsection{Rescaling the Frank--Wolfe recursion}\label{sec:rescale}
The first stage has placed all subsequent selected vertices in $F_*$.
We now prepare the second stage by transforming the averaging recursion
into additive updates. This subsection constructs the rescaling and
computes its factors. Section~\ref{sec:cycling} supplies the bound for
weighted additive updates; Section~\ref{sec:rate-proofs} applies it to
the tangential error in shifted coordinates, then converts the result
back to the original iterates.

Fix $\gamma_k=\alpha/(k+\beta)$ with $\beta\geq\alpha>0$.
We work from $k=1$ onward, where $\gamma_k<1$ in every case; the first
update affects only the initial state of the argument. Subtracting
$x^*$ from the FW update gives
\[
 x_{k+1}-x^*=(1-\gamma_k)(x_k-x^*)+\gamma_k(s_k-x^*).
\]
Choose a positive sequence $\rho_k$ so that
$\rho_{k+1}(1-\gamma_k)=\rho_k$, and put
$w_k=\rho_{k+1}\gamma_k$. Multiplication by $\rho_{k+1}$ cancels the
contraction factor $1-\gamma_k$ and yields
\begin{equation}\label{eq:rescale}
 \rho_{k+1}(1-\gamma_k)=\rho_k,\qquad
 z_k:=\rho_k(x_k-x^*),\qquad z_{k+1}=z_k+w_k(s_k-x^*).
\end{equation}
The rescaled error $z_k$ therefore evolves by adding a weighted
direction, as required by the perceptron theorem.

Iterating the recurrence for $\rho_k$ gives
\[
 \rho_k=\rho_1\prod_{j=1}^{k-1}\frac{j+\beta}{j+\beta-\alpha}
 \qquad(k\geq1).
\]
The Euler Gamma function provides a convenient closed form for this
product when $\alpha$ and $\beta$ are real. It is defined by
$\Gamma(t)=\int_0^\infty r^{t-1}e^{-r}\,dr$ for $t>0$ and satisfies
$\Gamma(t+1)=t\Gamma(t)$ and $\Gamma(m+1)=m!$ for integers $m\geq0$.
Using its recurrence, we may choose the normalization
\[
 \rho_k:=\frac{\Gamma(k+\beta)}
 {\Gamma(\alpha+1)\Gamma(k+\beta-\alpha)},\qquad
 w_k:=\frac{\Gamma(k+\beta)}
 {\Gamma(\alpha)\Gamma(k+\beta+1-\alpha)}
 \qquad(k\geq1).
\]
Multiplying both sequences by a positive constant preserves
\eqref{eq:rescale}; this choice makes the common schedules especially
simple. Harmonic steps $1/(k+1)$ give $\rho_k=k$ and $w_k=1$, so the
rescaled update adds one direction at a time. Steps $2/(k+2)$ give
$\rho_k=k(k+1)/2$ and $w_k=k+1$.

The Gamma recurrence and the asymptotic identity
$\Gamma(k+a)/\Gamma(k+b)\sim k^{a-b}$ give
\[
 \rho_k\sim\frac{k^\alpha}{\Gamma(\alpha+1)},\qquad
 w_k\sim\frac{k^{\alpha-1}}{\Gamma(\alpha)},\qquad
 \frac{w_k}{\rho_k}=\frac{\alpha}{k+\beta-\alpha},\qquad
 \frac{w_k}{w_{k+1}}=\frac{k+\beta+1-\alpha}{k+\beta}.
\]
Consequently the weights increase for $\alpha>1$, are constant for
$\alpha=1$, and decrease for $0<\alpha<1$. For $\alpha=1$ and any
$\beta\geq1$, the formulas reduce to $\rho_k=k+\beta-1$ and $w_k=1$.

These growth orders explain the exponent in Theorem~\ref{thm:app-fw}.
For the exact oracle, Section~\ref{sec:rate-proofs} will prove
$\norm{z_m}=O(1+W_m)$, where $W_m=\max_{1\leq k<m}w_k$.
The formulas above give $W_m=O(1+m^{(\alpha-1)_+})$, with
$(a)_+=\max\{a,0\}$. Dividing by $\rho_m\asymp m^\alpha$ then yields
\[
 \norm{x_m-x^*}
 =O\!\left(\frac{1+m^{(\alpha-1)_+}}{m^\alpha}\right)
 =O(m^{-\min\{\alpha,1\}}).
\]

\subsection{A bound for weighted perceptron recursions}\label{sec:cycling}
The rescaling gives the recursion~\eqref{eq:rescale} with additive updates and increments $s_k-x^*$.
After vertex-face identification these increments belong to the fixed
finite set $U_*:=\operatorname{vert}(F_*)-x^*$. The following bound is
the input for the third stage: after linearization and a fixed shift,
it controls the tangential error in Section~\ref{sec:rate-proofs}.

A perceptron recursion repeatedly adds a direction from a finite set,
chosen by minimizing a linear score of the current state. The following
weighted version allows an error $B$ in this minimization and weights
between zero and $W$. It bounds every state by the initial norm, the
largest weight, and the score error. This is the weighted perceptron
cycling theorem of \cite[Corollary~1]{RobustCycling}; ``cycling'' does
not assert that a trajectory is periodic. We use this theorem as an
external input and apply it below to the states up to any fixed index
to control the nonlinear Taylor remainder. The score tolerance is needed
even for exact FW oracles, because linearizing $\Phi$ changes their
scores.

\begin{theorem}[Weighted cycling {\cite[Corollary~1]{RobustCycling}}]
\label{thm:weighted}
Let $U$ be a nonempty finite subset of a finite-dimensional Euclidean
space $E$, with $0\in\conv U$, and let $A\colon E\to E$ be linear with
$\ip{Ah}h\geq\mu\norm h^2$ for all $h\in E$ and some $\mu>0$.
There is a constant $C_w=C_w(E,U,A)\geq1$ such that every
finite or infinite recursion
\[
 z_{k+1}=z_k+w_ku_k,\qquad u_k\in U,\qquad 0\leq w_k\leq W,
 \qquad
 \ip{Az_k}{u_k}\leq\min_{u\in U}\ip{Az_k}u+B,
\]
with $W,B\geq0$, satisfies
\begin{equation}\label{eq:weighted-bound}
 \sup_k\norm{z_k}\leq C_w(W+\norm{z_0}+B).
\end{equation}
The constant is independent of $W$, $B$, the initial state, the number of updates,
and the admissible choices. 
\end{theorem}
For the remainder of this subsection, fix these data and $C_w$.
For arbitrary nonnegative weights and score tolerances $b_k$, fix an
index $m\geq1$ and apply the theorem to the states $z_0,\ldots,z_m$.
The largest weight and tolerance used in these updates are
$W_m:=\max_{0\leq k<m}w_k$ and $B_m:=\max_{0\leq k<m}b_k$. Then
\begin{equation}\label{eq:horizon}
 \max_{0\leq j\leq m}\norm{z_j}
 \leq C_w(W_m+\norm{z_0}+B_m).
\end{equation}
For the nonlinear mapping, the score tolerance depends on the state being
bounded. The next lemma estimates all states with indices $N,\ldots,m$
when this dependence has a sufficiently small coefficient.

\begin{lemma}\label{lem:absorb}
Fix $m>N$. Suppose $y_{k+1}=y_k+w_ku_k$, $u_k\in U$, and $w_k\geq0$
for $N\leq k<m$, and the score tolerance at each such step is at most
$b+\eta\norm{y_k}$, where $b,\eta\geq0$ and $C_w\eta<1$. Put
$W_m:=\max_{N\leq k<m}w_k$. Then
\begin{equation}\label{eq:finite-absorb}
 \max_{N\leq k\leq m}\norm{y_k}
 \leq\frac{C_w\bigl(W_m+\norm{y_N}+b\bigr)}
 {1-C_w\eta}.
\end{equation}
\end{lemma}
\begin{proof}
Put $M_m:=\max_{N\leq k\leq m}\norm{y_k}$. Every score tolerance
for $N\leq k<m$ is at most $b+\eta M_m$. Theorem~\ref{thm:weighted} gives
\[
 M_m\leq C_w(W_m+\norm{y_N}+b)+C_w\eta M_m.
\]
Subtracting $C_w\eta M_m$ from both sides and dividing by the positive
number $1-C_w\eta$ proves the bound. If $W_m=0$, the sequence is
constant and the bound is immediate.
\end{proof}

\subsection{Proof of the last-iterate bound}\label{sec:rate-proofs}
We now combine the preceding steps. Uniform attraction identifies the
selected face vertices and makes the Taylor remainder small. Rescaling
and a tangential shift put the recursion in the form required by
Theorem~\ref{thm:weighted}. Lemma~\ref{lem:absorb} then bounds the shifted
states while accounting for the Taylor remainder.
Figure~\ref{fig:face-identification} illustrates the decomposition used below:
selected vertices enter the exposed face, while the iterates may remain outside it.
\input{face_identification.tex}
\begin{proof}[Proof of Theorem~\ref{thm:app-fw}]
The prescribed schedule satisfies the hypotheses of
Lemma~\ref{lem:uniform-fw}, and the solution is unique by
Lemma~\ref{lem:unique}. There is therefore a common identification index
$N_0\geq1$ such that all subsequent selected vertices lie in $F_*$.
We use the rescaling factors defined for $k\geq1$ in
Section~\ref{sec:rescale}.

Set $U_*:=\operatorname{vert}(F_*)-x^*$, so $0\in\conv U_*$, and use
$z_k=\rho_k(x_k-x^*)$. If $U_*=\{0\}$, \eqref{eq:rescale} makes $z_k$
constant after $N_0$, and $\norm{x_k-x^*}=O(k^{-\alpha})$, which suffices.
Otherwise put $R_*:=\max_{u\in U_*}\norm u>0$. Let $P_*$ be the orthogonal
projection onto $T_*$, and write
\[
 y_k=P_*z_k,\qquad q=z_k-y_k,\qquad
 A_*=P_*D\Phi(x^*)|_{T_*},\qquad c=P_*D\Phi(x^*)q.
\]
Every increment $w_k(s_k-x^*)$ belongs to $T_*$, so $q$ is constant after
$N_0$. By \eqref{eq:face-coercivity},
$\ip{A_*h}h\geq\mu_*\norm h^2$ for $h\in T_*$. Thus $A_*$ is invertible
and $\norm{A_*^{-1}}\leq\mu_*^{-1}$. Define
\[
 d=A_*^{-1}c,\qquad \widetilde y_k=y_k+d,\qquad u_k=s_k-x^*.
\]
Then $\widetilde y_{k+1}=\widetilde y_k+w_ku_k$. The Taylor remainder
\[
 r_k:=\rho_kP_*\bigl(\Phi(x_k)-\Phi(x^*)
              -D\Phi(x^*)(x_k-x^*)\bigr)
\]
obeys
\begin{equation}\label{eq:boundary-remainder}
 \norm{r_k}\leq\omega(\norm{x_k-x^*})\norm{z_k},\qquad \omega(t)\to0.
\end{equation}
A common nondecreasing modulus $\omega$ follows by taking the supremum of
the normalized Taylor remainder on $K\cap B(x^*,t)$.

For $u\in U_*$, $\ip{\Phi(x^*)}u=0$. Multiplying the oracle inequality
by $\rho_k>0$ therefore gives
\[
 \ip{A_*\widetilde y_k+r_k}{u_k}
 =\min_{u\in U_*}\ip{A_*\widetilde y_k+r_k}u.
\]
Comparing with a minimizer of the linearized score yields
\begin{equation}\label{eq:boundary-oracle-error}
 \ip{A_*\widetilde y_k}{u_k}
 \leq\min_{u\in U_*}\ip{A_*\widetilde y_k}u
       +2R_*\norm{r_k}.
\end{equation}
The substitution $\widetilde y_k=y_k+A_*^{-1}c$ gives
$A_*\widetilde y_k=A_*y_k+c$. The oracle tolerance in
\eqref{eq:boundary-oracle-error} therefore depends only on the Taylor
remainder $r_k$.

Let $C_{w,*}$ be the constant in Theorem~\ref{thm:weighted} for
$(T_*,U_*,A_*)$. Choose $\epsilon>0$ with $2C_{w,*}R_*\epsilon<1/2$.
Uniform attraction supplies a common $N\geq N_0$ after which the Taylor
modulus is at most $\epsilon$. Put
\[
 D_K:=\operatorname{diam}(K),\quad L_*:=\norm{D\Phi(x^*)},\quad
 Q_N:=\rho_ND_K(1+L_*/\mu_*).
\]
Uniformly over the family,
\[
 \norm{y_N},\norm q\leq\rho_ND_K,\quad
 \norm c\leq L_*\rho_ND_K,\quad
 \norm d\leq(L_*/\mu_*)\rho_ND_K.
\]
Hence $\norm{\widetilde y_N}\leq Q_N$ and
$\norm{z_k}\leq\norm{\widetilde y_k}+Q_N$ for $k\geq N$.
Fix any index $m>N$ and set
\[
 Y_m:=\max_{N\leq k\leq m}\norm{\widetilde y_k},\qquad
 W_m:=\max_{N\leq k<m}w_k.
\]
The tolerance is at most $2R_*\epsilon(Y_m+Q_N)$. Applying
Lemma~\ref{lem:absorb} with $b=2R_*\epsilon Q_N$ and
$\eta=2R_*\epsilon$ gives
\[
 Y_m\leq\frac{C_{w,*}(W_m+Q_N+2R_*\epsilon Q_N)}
                   {1-2C_{w,*}R_*\epsilon}.
\]
The coefficient in this bound is independent of $m$.
The schedule formulas in Section~\ref{sec:rescale} imply
\[
 W_m=O(1+m^{(\alpha-1)_+}),\qquad \rho_m\asymp m^\alpha,
\]
where $a_+:=\max\{a,0\}$. Adding $Q_N$ and dividing by $\rho_m$ gives
$\norm{x_m-x^*}=O(m^{-\min\{\alpha,1\}})$.
The diameter bound covers the finite prefix after enlarging the constant.
The constants may depend on the instance and schedule but are uniform
over starting points and vertex selections.
\end{proof}

\section{Step Sizes and Further Consequences}\label{sec:consequences}
The proof in Section~\ref{sec:rate-proofs} bounds the rescaled iterates
in terms of the largest weight used so far. This estimate also applies
to general step-size schedules. We first state the resulting bound,
then derive the gap bounds, and finally prove the lower bounds of
Theorem~\ref{thm:schedule-characterization} and their oracle-count
interpretation.
\subsection{General step-size schedules}
\begin{corollary}\label{cor:general-schedules}
Under the assumptions on $K$ and $\Phi$ of Theorem~\ref{thm:app-fw}, use any
fixed deterministic schedule with $0<\gamma_k\leq1$, $\gamma_k\to0$, and
$\sum_k\gamma_k=\infty$ in the exact vertex iteration. Choose $n\geq1$ so that
$\gamma_k<1$ for $k\geq n$, and define
\[
 \rho_n=1,\qquad \rho_{k+1}=\frac{\rho_k}{1-\gamma_k},\qquad
 w_k=\rho_{k+1}\gamma_k,\qquad W_m=\max_{n\leq k<m}w_k.
\]
There is a finite constant $C$, depending only on the problem, the fixed
schedule, and $n$, such that
\begin{equation}\label{eq:general-schedule-rate}
 \norm{x_m-x^*}\leq C\frac{W_m}{\rho_m}\qquad(m>n),
\end{equation}
uniformly over starting points and vertex tie-breaking. If the weights are
eventually nondecreasing, equivalently if eventually
\begin{equation}\label{eq:schedule-weight-monotonicity}
 \gamma_{k+1}\geq\frac{\gamma_k}{1+\gamma_k},
\end{equation}
then $\norm{x_m-x^*}=O(\gamma_m)$. In particular, this holds for
$\gamma_k=c/(k+b)^a$ with $0<a<1$, $c,b>0$, and $b^a\geq c$.
\end{corollary}
\begin{proof}
The rescaling gives $z_{k+1}=z_k+w_k(s_k-x^*)$ exactly.
Lemma~\ref{lem:uniform-fw} supplies uniform identification and a common index
$N\geq n$ where the Taylor modulus is small. Applying the same
decomposition, change of coordinates, and Lemma~\ref{lem:absorb} as in the
proof of Theorem~\ref{thm:app-fw} gives $\norm{z_m}=O(1+W_m)$;
the singleton-face case gives a constant $z_m$.
Since $W_m\geq w_n>0$, we have $1+W_m\leq(1+w_n^{-1})W_m$; dividing by
$\rho_m$ and enlarging $C$ for the finite prefix proves
\eqref{eq:general-schedule-rate}. Moreover,
\[
 \frac{w_{k+1}}{w_k}
 =\frac{\gamma_{k+1}}{\gamma_k(1-\gamma_{k+1})}.
\]
Thus \eqref{eq:schedule-weight-monotonicity} is precisely eventual
nondecrease of the weights, in which case $W_m=O(w_{m-1})$ and
\[
 \frac{W_m}{\rho_m}=O(\gamma_{m-1})=O(\gamma_m),\qquad
 \gamma_{m-1}\leq\frac{\gamma_m}{1-\gamma_m}
\]
for all sufficiently large $m$. For the stated power schedule,
$\gamma_{k+1}^{-1}-\gamma_k^{-1}
=c^{-1}((k+b+1)^a-(k+b)^a)\to0$, so the condition holds eventually.
\end{proof}

The ratio in \eqref{eq:general-schedule-rate} tends to zero for every
schedule in the corollary: $\rho_m\to\infty$, each fixed $w_k/\rho_m\to0$,
and $w_k/\rho_m\leq\gamma_k$ for $n\leq k<m$. Splitting the maximum into
a finite prefix and a tail proves the claim.

\subsection{Affine mappings and harmonic steps}
\begin{corollary}\label{cor:affine-harmonic}
Suppose $\alpha=\beta=1$ and $\Phi(x)=Lx+a$ is monotone on $K$ and
satisfies \eqref{eq:face-coercivity}. Then the exact-oracle conclusion of
Theorem~\ref{thm:app-fw}, including boundary solutions, follows using only
the exact unit-step cycling bound of \cite[Theorem~1, $B=0$]{RobustCycling}.
\end{corollary}
\begin{proof}
Uniform vertex-face identification is independent of the cycling theorem.
After its common index $k_0\geq1$, use the decomposition and shift from the
proof of Theorem~\ref{thm:app-fw}. Here $\rho_k=k$, $w_k=1$, and $r_k=0$ identically, so
\[
 \widetilde y_{k+1}=\widetilde y_k+u_k,\qquad
 u_k\in\argmin_{u\in U_*}\ip{A_*\widetilde y_k}u.
\]
Let $C_*$ be the unit-step cycling constant. Exact cycling bounds
$\widetilde y_k$ by $C_*(1+\norm{\widetilde y_{k_0}})$.
The initial shifted state, the shift $d$, and the normal component $q$ are
uniformly bounded in terms of $k_0D_K$ and $\norm L/\mu_*$. Hence $z_k$ is
uniformly bounded. The case $U_*=\{0\}$ is immediate because $z_k$ is
constant after $k_0$.
\end{proof}

For $\Phi(x)=x-\mu$ with $\mu\in K$, the preceding corollary reduces
to finite herding and its perceptron-cycling proof
\cite{GelfandEtAl2010,BachEtAl2012}. A positive-definite symmetric
quadratic discrepancy reduces to this case by a linear change of
coordinates. The reduction above also handles nonsymmetric affine mappings
and boundary solutions. The substitution $\widetilde y_k=y_k+A_*^{-1}c$
gives $A_*\widetilde y_k=A_*y_k+c$, so the oracle scores take the form
required by the exact cycling theorem.

\subsection{Frank--Wolfe and dual VI gaps}\label{sec:gap-rates}
The two gaps in \eqref{eq:vi-gaps} have different local behavior. The
following consequence of face coercivity controls the dual gap without
assuming global strong monotonicity.

\begin{lemma}[Quadratic growth and the dual gap]\label{lem:dual-growth}
Under the hypotheses of Theorem~\ref{thm:app-fw}, there is $\lambda>0$
such that, for every $s\in K$,
\begin{equation}\label{eq:radial-growth}
 \ip{\Phi(s)}{s-x^*}\geq\lambda\norm{s-x^*}^2.
\end{equation}
Consequently, for any Lipschitz constant $L$ of $\Phi$ on $K$,
\begin{equation}\label{eq:dual-gap-bound}
 0\leq G(x)\leq\ip{\Phi(x^*)}{x-x^*}
                 +\frac{L^2}{4\lambda}\norm{x-x^*}^2\qquad\forall x\in K.
\end{equation}
\end{lemma}
\begin{proof}
The singleton case is immediate. Write $a=\Phi(x^*)$ and
$H(s)=\ip{\Phi(s)}{s-x^*}$. Monotonicity and the VI give
$H(s)\geq\ip a{s-x^*}\geq0$. In fact $H(s)>0$ for $s\ne x^*$:
if equality held, $h=s-x^*$ would belong to $T_*$, while the
nondecreasing function $t\mapsto\ip{\Phi(x^*+th)}h$ would have zero
endpoint values and strictly positive derivative at zero, a contradiction.

If \eqref{eq:radial-growth} failed, compactness and this strict positivity
would give $s_j\to x^*$ with $s_j\ne x^*$ and
$H(s_j)/\norm{s_j-x^*}^2\to0$. Pass to a subsequence so that
$v_j=(s_j-x^*)/\norm{s_j-x^*}\to v$, with $\norm v=1$.
The polyhedral cone $\operatorname{cone}(K-x^*)$ is closed, so $v$ lies
in this cone. Also $\ip a v=0$, since
$0\leq\ip a{v_j}\leq H(s_j)/\norm{s_j-x^*}\to0$.
Thus $x^*+tv\in F_*$ for some $t>0$, and $v\in T_*$.
Differentiability at $x^*$ now gives
\[
 \frac{H(s_j)}{\norm{s_j-x^*}^2}
 =\frac{\ip a{s_j-x^*}}{\norm{s_j-x^*}^2}
    +\ip{D\Phi(x^*)v_j}{v_j}+o(1)
 \geq\mu_*+o(1),
\]
a contradiction. This proves \eqref{eq:radial-growth}.

For $x,s\in K$, put $e=x-x^*$ and $h=s-x^*$. Then
\[
 \ip{\Phi(s)}{x-s}
 =\ip a e+\ip{\Phi(s)-a}e-H(s)
 \leq\ip a e+L\norm h\norm e-\lambda\norm h^2
 \leq\ip a e+\frac{L^2}{4\lambda}\norm e^2.
\]
Maximizing over $s$ proves \eqref{eq:dual-gap-bound}.
\end{proof}

\begin{corollary}[Frank--Wolfe gap, dual gap, and squared distance]
\label{cor:gap-distance}
For exact-oracle trajectories under Theorem~\ref{thm:app-fw}, put
$r=\min\{\alpha,1\}$ and $s=\min\{\alpha,2\}$. Uniformly over starts
and vertex selections for the fixed instance and schedule,
\[
 V(x_k)=O(k^{-r}),\qquad G(x_k)=O(k^{-s}),\qquad
 \norm{x_k-x^*}^2=O(k^{-2r}).
\]
\end{corollary}
\begin{proof}
Let $D_K=\operatorname{diam}(K)$ and let $L$ be a Lipschitz constant for
$\Phi$ on $K$. The VI and Lipschitz continuity give
\[
 V(x)\leq(LD_K+\norm{\Phi(x^*)})\norm{x-x^*}.
\]
Indeed, split $\ip{\Phi(x)}{x-s}$ into
$\ip{\Phi(x)-\Phi(x^*)}{x-s}$,
$\ip{\Phi(x^*)}{x-x^*}$, and the nonpositive term
$\ip{\Phi(x^*)}{x^*-s}$. The distance theorem proves the bounds for $V$
and squared distance.
After the common face-identification index $N$, all selected vertices
belong to $F_*$. Hence $b_k:=\ip{\Phi(x^*)}{x_k-x^*}$ satisfies
\[
 b_{k+1}=(1-\gamma_k)b_k,\qquad
 b_k=\frac{\rho_N}{\rho_k}b_N=O(k^{-\alpha})\quad(k\geq N).
\]
Lemma~\ref{lem:dual-growth} yields
$G(x_k)=O(k^{-\alpha}+k^{-2r})=O(k^{-s})$.
The identification index is uniform and $b_N\leq\norm{\Phi(x^*)}D_K$,
so the constants are uniform as stated.
\end{proof}

For the general exact schedules of Corollary~\ref{cor:general-schedules},
the same argument gives
\[
 V(x_m)=O(W_m/\rho_m),\qquad
 G(x_m)=O\bigl(\rho_m^{-1}+(W_m/\rho_m)^2\bigr).
\]
Thus the schedule dependence of the dual gap differs from that of
distance even before specializing to reciprocal steps.

\subsection{Lower bounds and schedule optimality}\label{sec:optimality}
We now prove Theorem~\ref{thm:schedule-characterization}. The two
coordinates of one fixed quadratic problem isolate the oscillation and
memory mechanisms (Figure~\ref{fig:lower-bound}); together with
Corollary~\ref{cor:general-schedules} they characterize the rate for
every schedule with eventually monotone rescaled weights.

\input{lower_bound.tex}

\begin{proof}[Proof of Theorem~\ref{thm:schedule-characterization}]
On the square problem in the theorem, the solution is $x^*=(1/2,0)$.
Write $x_k=(\xi_k,\eta_k)$, $e_k=\norm{x_k-x^*}$, and the selected
vertex as $(a_k,b_k)$. The operator is affine, symmetric, and strongly
monotone with derivative $I$, and its oracle
separates by coordinate. If no step equals one, take $\eta_0=1$.
Then all second-coordinate selections are zero and
$\eta_n=\prod_{j=0}^{n-1}(1-\gamma_j)>0$.
Otherwise let $h<n$ be the last index of a unit step. Start with
$\eta_0=0$, select $b_k=0$ for $k<h$, and select $b_h=1$; these selections
are exact because the second-coordinate score is zero. Then
$\eta_{h+1}=1$, and all later second-coordinate selections are uniquely
zero. In either case $c_n:=\eta_n>0$ depends only on the schedule and
$n$, and
\begin{equation}\label{eq:pointwise-lower}
 \eta_k=c_n\prod_{j=n}^{k-1}(1-\gamma_j)
       =c_n\rho_k^{-1},\qquad
 e_k\geq c_n\rho_k^{-1}\qquad(k\geq n).
\end{equation}
Choose any exact first-coordinate selections. Lemma~\ref{lem:uniform-fw}
gives $\xi_k\to1/2$. Since $a_k\in\{0,1\}$, eventually
$|a_k-\xi_k|\geq1/4$, whence
\begin{equation}\label{eq:oscillation-lower}
 \max\{e_k,e_{k+1}\}\geq\tfrac12|\xi_{k+1}-\xi_k|
 =\tfrac12\gamma_k|a_k-\xi_k|\geq\gamma_k/8.
\end{equation}
With $c=\min\{c_n,1/8\}>0$, the two bounds give
$\Delta_k\geq c\max\{\gamma_k,\rho_k^{-1}\}=cq_k$ eventually.
If $e_k=o(k^{-1})$, the last inequality
would imply $\gamma_k=o(k^{-1})$. For every $\varepsilon>0$, eventually
$\gamma_j\leq\varepsilon/(2(j+1))$ and $\gamma_j\leq1/2$.
The inequality $\log(1-t)\geq-2t$ on $[0,1/2]$ then gives
$\eta_k\geq c_\varepsilon k^{-\varepsilon}$, with
$c_\varepsilon>0$. Taking $\varepsilon=1/2$ contradicts $e_k=o(k^{-1})$.
Together with the harmonic upper
bound in Theorem~\ref{thm:app-fw}, this proves the stated optimality.
\end{proof}

\paragraph{Matching upper bounds for monotone weights.}\label{par:matching-upper-bounds}
We now justify the comparison stated after
Theorem~\ref{thm:schedule-characterization}. For $m>n$ let
$W_m=\max_{n\leq j<m}w_j$ as in Corollary~\ref{cor:general-schedules}.
If the weights are nondecreasing after $N$, they are bounded below by
$w_N>0$. The finitely many earlier weights are bounded, so
$W_{k+1}=O(w_k)$. Since $\gamma_k\to0$,
\[
 \Delta_k\leq C\frac{W_{k+1}}{\rho_k}
 =O\!\left(\frac{w_k}{\rho_k}\right)
 =O(\gamma_k).
\]
Here $\rho_k^{-1}=O(\gamma_k)$ follows from
$\gamma_k\rho_k=w_k(1-\gamma_k)\geq w_N/2$ eventually.
If the weights are eventually nonincreasing, then $\sup_{k\geq n}w_k<\infty$.
The corollary gives $\Delta_k=O(\rho_k^{-1})$, and
$\gamma_k=w_k/\rho_{k+1}=O(\rho_k^{-1})$.
Thus in both cases $\Delta_k=O(q_k)$, matching the lower order in
Theorem~\ref{thm:schedule-characterization}.

With eventually nondecreasing weights one has
$W_m/\rho_m\asymp\gamma_{m-1}$, with equality when the weights are
nondecreasing from $n$ on; for eventually nonincreasing weights,
$W_m/\rho_m\asymp\rho_m^{-1}$. If the steps have upward jumps,
$\gamma_{m-1}$ need not be comparable to $\gamma_m$; this index
distinction matters when passing from a one-iterate bound to $\Delta_k$.
The retained-weight term in $q_k=\max\{\gamma_k,\rho_k^{-1}\}$ has a
separate role: it accounts for memory. Even without jumps, reciprocal
steps with $0<\alpha<1$ have $\rho_k^{-1}\asymp k^{-\alpha}$, which
dominates $\gamma_k\asymp k^{-1}$.

\paragraph{Tight reciprocal-step exponents.}
For $\gamma_k=\alpha/(k+\beta)$ the weights are monotone and
$\rho_k^{-1}\asymp k^{-\alpha}$. Theorem~\ref{thm:schedule-characterization}
therefore gives the tight order $k^{-\min\{\alpha,1\}}$ on the same
two-dimensional example. For $0<\alpha<1$, its vertical coordinate alone
has this order at every sufficiently large iterate. For $\alpha\geq1$,
the oscillation bound excludes $o(k^{-1})$.
In the harmonic one-dimensional example $K=[0,1]$, $\Phi(x)=x-1/2$,
one also has $x_T=m_T/T$ with $m_T$ integral, so the error is at least
$1/(2T)$ at every odd horizon. Figure~\ref{fig:schedule-rates} illustrates
the two mechanisms on the fixed two-dimensional problem.

\input{schedule_rates.tex}

\paragraph{Selecting the schedule: rate order and constants.}
The factor in Corollary~\ref{cor:general-schedules} also explains why
harmonic steps are a natural choice. Its normalization gives
$\rho_{k+1}-\rho_k=w_k$ and hence
\[
 \rho_m=1+\sum_{k=n}^{m-1}w_k,\qquad
 \frac{W_m}{\rho_m}\geq
 \frac{1-\rho_m^{-1}}{m-n},\qquad
 \liminf_{m\to\infty}m\frac{W_m}{\rho_m}\geq1.
\]
For $\gamma_k=\alpha/(k+\beta)$ with $\alpha\geq1$, the weights are
nondecreasing and the factor is exactly
\[
 \frac{W_m}{\rho_m}=\frac{w_{m-1}}{\rho_m}
 =\gamma_{m-1}=\frac{\alpha}{m-1+\beta}.
\]
Thus $\alpha=1$ attains the smallest possible asymptotic coefficient of
this schedule factor. All $\alpha\geq1$ attain the optimal distance order;
$0<\alpha<1$ gives the slower tight order $k^{-\alpha}$.
Likewise, the bound $O(k^{-a})$ for $\gamma_k=c/(k+b)^a$, $0<a<1$,
cannot be improved to $o(k^{-a})$, since the same oscillation argument
gives the adjacent-iterate lower bound $\gamma_k/8$.
This comparison concerns the schedule factor only. The full constant in
\eqref{eq:general-schedule-rate} also depends on the schedule through the
index after which all trajectories are sufficiently close to $x^*$, so it does not show that
$1/(k+1)$ has a smaller worst-case constant or a shorter transient than
$2/(k+2)$, and the lower bounds say nothing about adaptive step rules.

\begin{proposition}[Tight gap exponents]\label{prop:gap-sharpness}
For reciprocal steps $\alpha/(k+\beta)$, the exponents
$\min\{\alpha,1\}$ for $V$ and $\min\{\alpha,2\}$ for $G$ in
Corollary~\ref{cor:gap-distance} are tight in the worst case, measured
over two consecutive iterates. For every positive, vanishing,
nonsummable prescribed schedule, a trajectory on the square problem of
Theorem~\ref{thm:schedule-characterization} has
$V(x_k)\ne o(k^{-1})$ and $G(x_k)\ne o(k^{-2})$.
\end{proposition}
\begin{proof}
On that square problem, write $\delta=\xi-1/2$. Direct maximization gives
\begin{equation}\label{eq:square-gaps}
 V(\xi,\eta)=\delta^2+\tfrac12|\delta|+\eta^2,\qquad
 G(\xi,\eta)=\tfrac14(\delta^2+\eta^2).
\end{equation}
For $G$, the maximizing test point is the midpoint of $(\xi,\eta)$ and
$x^*$. The horizontal-increment argument leading to
\eqref{eq:oscillation-lower} gives
$\max\{|\delta_k|,|\delta_{k+1}|\}\geq\gamma_k/8$ eventually. Thus
\[
 \max\{V(x_k),V(x_{k+1})\}\geq\gamma_k/16,\qquad
 \max\{G(x_k),G(x_{k+1})\}\geq\gamma_k^2/256.
\]
The equality for $G$ and Theorem~\ref{thm:schedule-characterization}
exclude $G(x_k)=o(k^{-2})$. If $V(x_k)=o(k^{-1})$, the first inequality
would give $\gamma_k=o(k^{-1})$. The product estimate in the proof of
that theorem gives $\eta_k\geq c k^{-1/4}$ eventually, contradicting
$V(x_k)\geq\eta_k^2=o(k^{-1})$.

To realize the retained-weight obstruction linearly in both gaps, use
another fixed instance on $[0,1]^2$:
\[
 \widehat\Phi(\xi,\eta)
 =\bigl(\xi-\tfrac12-\eta,\ \xi-\tfrac12+\eta+\tfrac14\bigr),
 \qquad x^*=(\tfrac12,0).
\]
Its derivative has symmetric part $I$, so it is strongly monotone, and
$\widehat\Phi(x^*)=(0,1/4)$ verifies the VI. Testing the dual gap at
$x^*$ shows $V(\xi,\eta)\geq G(\xi,\eta)\geq\eta/4$.
If $\beta>\alpha$, start with $\eta_0=1$. If $\beta=\alpha$, start at
$(0,0)$; the unique first oracle vertex is $(1,1)$, so $\eta_1=1$.
In both cases $\eta_1>0$: when $\beta>\alpha$,
$\eta_1\geq1-\gamma_0>0$, and when $\beta=\alpha$, $\eta_1=1$.
All steps from index one onward are below one, and all selected vertical
coordinates are nonnegative. Consequently
\[
 \eta_k\geq\eta_1\prod_{j=1}^{k-1}(1-\gamma_j)\asymp k^{-\alpha}.
\]
This supplies the lower exponent $\alpha$ when $\alpha\leq1$ for $V$
and when $\alpha\leq2$ for $G$. For the remaining ranges, the square
oscillation bounds give exponents one and two, respectively.
\end{proof}

\subsection{Oracle counts for VI-gap accuracy}\label{sec:oracle-counts}
We now translate the gap rates and their adjacent-iterate lower bounds
into oracle counts. An exact update uses one evaluation of $\Phi$ and one
vertex-returning LMO call, and the Frank--Wolfe gap comes for free from
the same pair, $V(x_k)=\ip{\Phi(x_k)}{x_k-s_k}$, so it serves as a
stopping certificate. Evaluating $G$ would require a different
maximization; our bound on $G(x_k)$ is a guarantee on the returned
iterate, not a procedure for computing $G$. Because the lower bounds hold
over two consecutive iterates, the natural count is the index of
\emph{sustained gap accuracy}: for $R\in\{V,G\}$,
\begin{equation}\label{eq:sustained-accuracy}
 N_{\mathrm{tail}}^R(\varepsilon;x)
 :=\min\{N\in\mathbb N_0:\ R(x_k)\leq\varepsilon
                              \text{ for every }k\geq N\}.
\end{equation}
Reaching $x_N$ takes $N$ calls to each oracle, and evaluating $V(x_N)$
costs at most one additional pair. The upper bounds below also bound the
first accurate iterate; the matching lower bounds concern sustained
accuracy rather than the first successful stopping certificate.

\begin{corollary}[Oracle counts]
\label{cor:oracle-tail}
Fix an instance satisfying Theorem~\ref{thm:app-fw} and reciprocal steps
$\gamma_k=\alpha/(k+\beta)$, $\beta\geq\alpha>0$. Put
$r=\min\{\alpha,1\}$ and $s=\min\{\alpha,2\}$. Then
\[
 N_{\mathrm{tail}}^V(\varepsilon;x)=O(\varepsilon^{-1/r}),\qquad
 N_{\mathrm{tail}}^G(\varepsilon;x)=O(\varepsilon^{-1/s})
 \qquad(\varepsilon\downarrow0),
\]
uniformly over starts and vertex selections for the fixed instance and
schedule. Each order is attained on one of the fixed two-dimensional
examples in Proposition~\ref{prop:gap-sharpness}.
For every positive, vanishing, nonsummable prescribed schedule, a
trajectory on the square problem has
$N_{\mathrm{tail}}^V(\varepsilon;x)\ne o(\varepsilon^{-1})$ and
$N_{\mathrm{tail}}^G(\varepsilon;x)\ne o(\varepsilon^{-1/2})$.
Thus $\alpha\geq1$ is asymptotically order-optimal for $V$, and
$\alpha\geq2$ is asymptotically order-optimal for $G$, within these
prescribed vertex updates and this sustained-accuracy criterion.
\end{corollary}
\begin{proof}
For either residual, write $R(x_k)\leq Ck^{-a}$, with $a=r$ for $V$
and $a=s$ for $G$, using Corollary~\ref{cor:gap-distance}. Then
$N_{\mathrm{tail}}^R(\varepsilon;x)\leq
\max\{1,\lceil(C/\varepsilon)^{1/a}\rceil\}$.
On the matching example, Proposition~\ref{prop:gap-sharpness} gives
$\max\{R(x_k),R(x_{k+1})\}\geq ck^{-a}$ eventually. Since the
corresponding tail index $N$ tends to infinity, both residuals at
$N,N+1$ are at most $\varepsilon$, and
$\varepsilon\geq cN^{-a}$ gives the matching lower order.
Finally, for any residual sequence and $a>0$,
$N_{\mathrm{tail}}^R(\varepsilon;x)=o(\varepsilon^{-1/a})$
would imply $R(x_k)=o(k^{-a})$: for each $b>0$, substitute
$\varepsilon=bk^{-a}$ to obtain $N_{\mathrm{tail}}^R(bk^{-a};x)<k$
eventually and hence $k^aR(x_k)\leq b$.
Apply this with $a=1$ and $a=2$ to the last assertion of
Proposition~\ref{prop:gap-sharpness}.
\end{proof}

Distance accuracy $\norm{x_k-x^*}\leq\varepsilon$ has its own tight tail
count $O(\varepsilon^{-1/r})$ by
Theorems~\ref{thm:app-fw} and~\ref{thm:schedule-characterization}; it is
a different criterion from VI-gap accuracy. These oracle bounds have
instance-dependent constants that may also
depend on the schedule. They do not establish minimax oracle complexities
over a normalized problem class, as the following classical obstruction
makes explicit.

\begin{proposition}[Standard simplex obstruction for VI-gap accuracy]
\label{prop:simplex-oracle}
For $d\geq2$, let
\[
 K_d=\conv\{e_1,\ldots,e_d\}\subset\R^d,\qquad
 u_d=d^{-1}(1,\ldots,1),\qquad \Phi_d(x)=x-u_d,
\]
where $e_i$ are the standard basis vectors. The solution is $u_d$, the
Lipschitz and strong-monotonicity constants are one, and
$\operatorname{diam}(K_d)=\sqrt2$.
Consider any method given a starting vertex and returning, after $t$
vertex-LMO calls, a point $y_t$ in the convex hull of that vertex and
the returned vertices. Then
\begin{align}
 \norm{y_t-u_d}^2&\geq\frac1{t+1}-\frac1d,
                         \label{eq:simplex-distance-lower}\\
 V(y_t)\geq G(y_t)&=\tfrac14\norm{y_t-u_d}^2
       \geq\tfrac14\left(\frac1{t+1}-\frac1d\right).
                         \label{eq:simplex-gap-lower}
\end{align}
For either gap criterion with tolerance $\varepsilon$, necessarily
$t+1\geq(4\varepsilon+1/d)^{-1}$. Taking $d\geq\varepsilon^{-1}$
gives a lower bound $(5\varepsilon)^{-1}-1$ on vertex-LMO calls.
\end{proposition}
\begin{proof}[Proof (the sparsity argument of {\cite[Theorem~1]{Lan2013}})]
The output has at most $t+1$ nonzero coordinates. Nonnegativity and
unit coordinate sum give $\norm{y_t}^2\geq1/(t+1)$ by
Cauchy--Schwarz, while
$\norm{y_t-u_d}^2=\norm{y_t}^2-1/d$.
For any $y\in K_d$, completing the square gives
\[
 G(y)=\max_{z\in K_d}\ip{z-u_d}{y-z}
 =\tfrac14\norm{y-u_d}^2,
\]
because the maximizing point $(y+u_d)/2$ belongs to $K_d$.
Use $V\geq G$ and rearrange to obtain the oracle lower bound.
\end{proof}

Thus the sparsity obstruction is $\Omega(\varepsilon^{-1})$ for VI-gap
accuracy when the dimension may grow, and $\Omega(\varepsilon^{-2})$ for
distance accuracy by \eqref{eq:simplex-distance-lower}. It applies to
adaptive queries and arbitrary convex combinations of acquired vertices,
and it counts vertex-LMO calls only. There is no conflict with the
$O(\varepsilon^{-1/2})$ dual-gap count of
Corollary~\ref{cor:oracle-tail}, whose constants may deteriorate as $K_d$
varies; a class-uniform upper bound with controlled geometric constants
would be needed before asserting minimax optimality.

\section{Extensions, Limitations and Open Problems}\label{sec:outlook}
\subsection{Inexact oracles and robustness}\label{sec:inexact-extensions}
The preceding results concern exact vertex oracles. The same argument
also yields quantitative extensions to vanishing oracle errors and
bounded delays.

\paragraph{Vanishing oracle errors.}
The analysis tolerates vertices with an additive score error:
\begin{equation}\label{eq:inexact-fw}
 \begin{split}
 x_{k+1}&=x_k+\gamma_k(s_k-x_k),\qquad s_k\in\operatorname{vert}(K),\\
 \ip{\Phi(x_k)}{s_k}&\leq\min_{s\in K}\ip{\Phi(x_k)}s+\delta_k.
 \end{split}
\end{equation}
\begin{theorem}[Inexact vertex oracles]\label{thm:inexact}
Under the assumptions of Theorem~\ref{thm:app-fw}, fix $D\geq0$ and $p>0$.
There is a finite constant $C=C(K,\Phi,\alpha,\beta,D,p)$ such that every trajectory of
\eqref{eq:inexact-fw} with $0\leq\delta_k\leq D(k+1)^{-p}$ satisfies
\[
 \norm{x_T-x^*}\leq C T^{-\min\{\alpha,1,p\}}\qquad(T\geq1).
\]
The bound is uniform over starting points, admissible tolerances, and all
admissible vertex selections. In particular, $p\geq\min\{\alpha,1\}$
preserves the exact-oracle exponent.
\end{theorem}

A proof is given in Appendix~\ref{app:inexact-proof}.

\paragraph{Gap rates with inexact oracles.}
Under Theorem~\ref{thm:inexact}, put $r=\min\{\alpha,1,p\}$.
Then, uniformly over the admissible trajectories,
\[
 V(x_k)=O(k^{-r}),\qquad
 G(x_k)=O(k^{-\min\{\alpha,2,2p\}}),\qquad
 \norm{x_k-x^*}^2=O(k^{-2r}).
\]
These bounds are proved in Appendix~\ref{app:inexact-gaps}.

\paragraph{Bounded delays and score-evaluation errors.}
Errors in the evaluation of $\Phi$ and delayed scores can be reduced
to the preceding additive oracle-error model. For reciprocal steps,
a bounded delay and evaluation errors of order $k^{-1}$ preserve the
exact-oracle distance exponent.

\begin{corollary}\label{cor:delays}
Under the assumptions of Theorem~\ref{thm:app-fw}, fix an integer $\tau\geq0$
and $E_0,H_0\geq0$. Suppose the iteration selects a vertex $s_k$ using
indices $j_k$ and score vectors $\widehat\Phi_k$ satisfying
\[
 \max\{0,k-\tau\}\leq j_k\leq k,\qquad
 \norm{\widehat\Phi_k-\Phi(x_{j_k})}\leq\frac{E_0}{k+1},
\]
and
\[
 \ip{\widehat\Phi_k}{s_k}
 \leq\min_{s\in K}\ip{\widehat\Phi_k}s+\frac{H_0}{k+1}.
\]
Then $\norm{x_T-x^*}=O(T^{-\min\{\alpha,1\}})$, uniformly over all
admissible delays, errors, starting points, and vertex selections.
\end{corollary}
A proof is given in Appendix~\ref{app:delays-proof}.

\paragraph{Stochastic feedback.}
Bo\c{t}, Mertikopoulos, Staudigl, and Vuong \cite{BotEtAl2021} prove
almost-sure convergence for a minibatch forward-backward-forward method
under pseudomonotonicity, using one projection per iteration. Our
robustness bounds apply when the oracle errors satisfy a common decaying
envelope. Establishing comparable FW distance rates under unbiased
stochastic feedback without such a pathwise envelope remains open here.

\subsection{Limitations}
The proof does not provide explicit bounds on the constants.
Uniform vertex-face identification and
entry into a neighborhood where the Taylor remainder is small enough for
Lemma~\ref{lem:absorb} are
obtained by compactness, and the cycling coefficient comes from an
existence theorem \cite{RobustCycling}. The bounds therefore give the
dependence on the accuracy for each fixed instance, not an explicit oracle
budget from the initial point, and the uniformity is over all admissible
trajectories of a fixed instance, schedule, and error envelope rather than
over a class of feasible sets and mappings.

The rate theorems retain the polytopal setting, the vertex restriction,
and the exposed-face coercivity assumption. The inexact result requires a
common decaying error envelope; it does not assert the same rate for
persistent positive errors or unbounded delays, and the dependence of its
constant on $D$ and $p$ is not claimed to be linear in $D$.

\subsection{Open problems}
\paragraph{Effective constants and step-size design.}
An effective bound for the cycling constant in terms of the direction
set and the operator would clarify the geometric dependence of the rate.
For the nonlinear FW theorem, one must also quantify attraction, vertex
identification, and entry into a neighborhood where the Taylor modulus
is small enough. This would turn the estimates
after entry into an explicit complexity bound from the initial iterate.
Corollary~\ref{cor:gap-distance} already separates the dual-gap exponents
of $1/(k+1)$ and $2/(k+2)$; effective estimates would also quantify their
constant and transient tradeoffs. A further question is whether adaptive step rules can improve
distance guarantees under the same VI assumptions while retaining a
single vertex-oracle update.

\paragraph{Minimax oracle complexity.}
A class-uniform comparison requires quantitative control of dimension,
polytope geometry, operator regularity, and coercivity, as well as a fixed
output and accuracy criterion. Proposition~\ref{prop:simplex-oracle}
gives a dimension-varying $\Omega(\varepsilon^{-1})$ lower bound for
VI-gap accuracy in the vertex-generated output model, even with
diameter, Lipschitz, and strong-monotonicity constants controlled.
The faster dual-gap rate has constants that depend on the instance
geometry.
Determining matching upper and lower oracle bounds under suitable
geometric assumptions, and the benefit of adaptive steps or other
output rules for nonlinear VIs, remains open here.

\paragraph{Weaker assumptions and other oracle outputs.}
It is natural to ask whether an error bound relative to the VI solution
set can replace coercivity on the exposed face, and what distance rates
remain possible when solutions are not unique. The growth-condition
and error-bound approaches in convex optimization
\cite{WirthPenaPokutta2025,WirthPenaPokuttaPolytopes2026} suggest a useful
comparison, but a nonsymmetric VI requires a separate argument.
Another question is whether the vertex restriction can be relaxed,
including for boundary solutions and approximate oracles returning
arbitrary points of $K$. The finite-direction cycling argument used here
does not directly cover such outputs.

\paragraph{Beyond polytopes.}
The uniform attraction lemma already applies to arbitrary compact convex
sets, but the rate proof uses a finite set of oracle directions.
Extending the last-iterate distance guarantee therefore requires a
replacement for this discrete geometric control. Possible settings
include solutions with a positive relative inradius, or assumptions that
give regularity of the linear minimization oracle, as in
\cite[Section~3.2]{Hough2026}. Boundary solutions require particular care
because the centered relative inradius then vanishes. We leave such
extensions outside the results proved here.

\section*{AI Disclosure}
Following the
\emph{Leiden Declaration on Artificial Intelligence and Mathematics}
by Alper et al.~\cite{LeidenDeclaration2026}, the author discloses discussions
with OpenAI's ChatGPT Astra and Anthropic's Claude Fable 5.1 and the use of OpenAI's Codex
for restructuring and editing the manuscript, checking
and tightening proofs, drafting parts of the Tikz figures, and suggesting related work. The
author verified all such output.
The research ideas, the method, and the main results are the authors'
own, and the author takes full responsibility for the content of this work.

\bibliographystyle{abbrvurl}
\begingroup
\small
\raggedright
\mathchardef\UrlBreakPenalty=100
\mathchardef\UrlBigBreakPenalty=100
\bibliography{master-bib}
\endgroup

\clearpage
\appendix
\section{Differential-Inclusion Details}\label{sec:compactness}
\subsection{Gap decay along the continuous dynamics}\label{sec:gap-decay}
We verify \eqref{eq:hough-gap-decay} for an absolutely continuous
$K$-valued solution of $\dot x\in\mathcal B(\Phi(x))-x$. Put
$q(x):=\min_{s\in K}\ip{\Phi(x)}{s}$ and write
$\dot x(t)=s(t)-x(t)$ with $s(t)\in\mathcal B(\Phi(x(t)))$.  At almost every time,
both $x$ and the Lipschitz function $q\circ x$ are differentiable.  At each
such time, put $p(t):=\Phi(x(t))$ and
$\dot p(t)=D\Phi(x(t))\dot x(t)$. The two one-sided envelope formulas are
\begin{align*}
 \lim_{h\downarrow0}\frac{q(x(t+h))-q(x(t))}{h}
 &=\min_{r\in\mathcal B(p(t))}\ip{\dot p(t)}r,\\
 \lim_{h\uparrow0}\frac{q(x(t+h))-q(x(t))}{h}
 &=\max_{r\in\mathcal B(p(t))}\ip{\dot p(t)}r.
\end{align*}
For the first formula, compare with a minimizer at $t$ for the upper bound,
and take a convergent subsequence of minimizers at $t+h$ for the lower
bound; compactness puts each limit in $\mathcal B(p(t))$. The expansion
$p(t+h)=p(t)+h\dot p(t)+o(|h|)$ is uniform when paired with $r\in K$.
The second formula follows in the same way with $h<0$, which reverses the
inequalities. Differentiability of $q\circ x$ makes the two limits equal.
Thus $\ip{\dot p(t)}r$ is the same for every $r\in\mathcal B(p(t))$.
Taking $r=s(t)$ and differentiating
$V(x(t))=\ip{p(t)}{x(t)}-q(x(t))$ gives
\[
 \frac{d}{dt}V(x(t))
 =-\ip{D\Phi(x(t))\dot x(t)}{\dot x(t)}-V(x(t))
 \leq -V(x(t))
\]
for almost every $t$.  The inequality follows from monotonicity of $\Phi$
by differentiating the monotonicity inequality along the feasible segment
from $x(t)$ to $s(t)$.  Absolute continuity then yields
\eqref{eq:hough-gap-decay}.

\subsection{Compactness for the interpolated dynamics}
The following is the standard compactness argument underlying limits of
time-shifted perturbed solutions; compare
\cite[proof of Theorem~4.2]{BenaimHofbauerSorin2005}. We state and prove
the version used in Section~\ref{sec:attraction} for families of
approximate trajectories on a compact domain.

\begin{theorem}[Compactness for approximate differential inclusions]
\label{thm:appendix-di-compactness}
Let $X\subset E$ be compact, and let
$G\colon X\rightrightarrows E$ be upper semicontinuous with nonempty compact
convex values.  Let $I=[a,b]$, and suppose that $f_n\colon I\to X$ are
absolutely continuous with a common Lipschitz constant.  Suppose there are
measurable maps $\widehat f_n\colon I\to X$ such that
\[
 \sup_{t\in I}\norm{f_n(t)-\widehat f_n(t)}\longrightarrow0
 \qquad(n\to\infty),
 \qquad
 \dot f_n(t)\in G(\widehat f_n(t))
 \quad\text{for a.e. }t\in I.
\]
Then there is a subsequence $(f_{n_j})_{j\geq1}$ and an absolutely continuous
function $f\colon I\to X$ such that
\[
 \sup_{t\in I}\norm{f_{n_j}(t)-f(t)}\longrightarrow0
 \qquad (j\to\infty)
\]
and
\[
 \dot f(t)\in G(f(t))
 \quad\text{for a.e. }t\in I.
\]
More generally, suppose that $f_n$ and $\widehat f_n$ are defined on intervals
$I_n$ that exhaust $\R$, in the sense that every compact interval is contained
in $I_n$ for all sufficiently large $n$, and that the preceding assumptions
hold on every compact interval, with the approximation error tending to zero
as $n\to\infty$.  Then there is a subsequence $(f_{n_j})_{j\geq1}$ and a
solution $f\colon\R\to X$ such that $f_{n_j}\to f$ uniformly on every compact
interval.
\end{theorem}

\begin{proof}
Arzel\`a--Ascoli gives indices $(n_j)$ for which $f_{n_j}$ converges uniformly.
We pass to this subsequence and suppress the subscript $n_j$ in the remainder
of the proof.  The derivatives
$g_n:=\dot f_n$ are bounded in $L^2(I;E)$ by the common Lipschitz constant, so,
after another subsequence, they converge weakly to some $g$.  Passing to the
limit in
\[
 f_n(t)-f_n(s)=\int_s^t g_n(r)\,dr
\]
shows that $f(t)-f(s)=\int_s^t g(r)\,dr$ and hence $g=\dot f$ almost
everywhere.

It remains to pass to the inclusion.  For $p\in E$, let
$h_p(x):=\max_{v\in G(x)}\ip{p}{v}$.  Upper semicontinuity and compactness of
the values imply that $h_p$ is upper semicontinuous.  For $\varepsilon>0$
define its $\varepsilon$-enlargement by
\[
 h_p^\varepsilon(x):=\sup\left\{
 h_p(y):y\in X,\ \norm{y-x}\leq\varepsilon
 \right\}.
\]
The function $h_p^\varepsilon$ is upper semicontinuous, hence measurable.
Indeed, if $x_\ell\to x$, pass to a subsequence realizing the upper limit
and choose $y_\ell\in X$ within $\varepsilon$ of $x_\ell$ and within
$1/\ell$ of the defining supremum.  A further convergent subsequence of
$(y_\ell)$ and upper semicontinuity of $h_p$ give
$\limsup_\ell h_p^\varepsilon(x_\ell)\leq h_p^\varepsilon(x)$.

Fix $\varepsilon>0$.  Since $\widehat f_n\to f$ uniformly, for all
sufficiently large $n$ and almost every $t\in I$,
\[
 \ip{p}{g_n(t)}
 \leq h_p(\widehat f_n(t))
 \leq h_p^\varepsilon(f(t)).
\]
Consequently, for every nonnegative bounded measurable function $\varphi$,
weak convergence yields
\[
 \int_I\varphi(t)\ip{p}{g(t)}\,dt
 =\lim_{n\to\infty}\int_I\varphi(t)\ip{p}{g_n(t)}\,dt
 \leq\int_I\varphi(t)h_p^\varepsilon(f(t))\,dt.
\]
It follows that
$\ip{p}{g(t)}\leq h_p^\varepsilon(f(t))$ almost everywhere.  Apply this
conclusion with $p$ in a fixed countable dense subset of $E$ and with
$\varepsilon$ in the positive rationals.  Outside one null set, all these
inequalities hold simultaneously.  Moreover, upper semicontinuity of $h_p$
and compactness of $X$ imply
\[
 h_p^\varepsilon(x)\downarrow h_p(x)
 \qquad(\varepsilon\downarrow0).
\]
Indeed, $h_p^\varepsilon(x)\geq h_p(x)$, while maximizers corresponding to
any sequence $\varepsilon_\ell\downarrow0$ converge to $x$ and upper
semicontinuity gives the reverse inequality in the limit.
Thus $\ip{p}{g(t)}\leq h_p(f(t))$ for all $p$ in the countable dense set and
almost every $t$.  Continuity in $p$ extends the inequalities to every
$p\in E$, and the support-function characterization of a closed convex set
gives $g(t)\in G(f(t))$ almost everywhere.

For intervals exhausting $\mathbb R$, apply the preceding argument on
$[-m,m]$ and take a diagonal subsequence.  This is the finite-dimensional
convergence theorem for differential inclusions; see
\cite[Theorem~1.1.4]{AubinCellina1984}.
\end{proof}

\section{Proofs for the Extensions}\label{app:extension-proofs}
This appendix proves the extensions stated in
Section~\ref{sec:inexact-extensions}.

\subsection{Inexact vertex oracles}\label{app:inexact-proof}
\begin{proof}[Proof of Theorem~\ref{thm:inexact}]
First, the uniform attraction and vertex-face identification in
Lemma~\ref{lem:uniform-fw} extend to any fixed nonnegative envelope
$e_k\to0$. To see this, define
\[
 \mathcal B_e(p):=\{s\in K:\ip ps\leq\min_{v\in K}\ip pv+e\},
\]
and let $\mathfrak X_e$ consist of all trajectories with $x_0\in K$,
$x_{k+1}=x_k+\gamma_k(s_k-x_k)$, and
$s_k\in\mathcal B_{e_k}(\Phi(x_k))$. In the contradiction argument of
Lemma~\ref{lem:uniform-fw}, the interpolations now satisfy
\[
 \dot W_j(t)\in
 \mathcal B_{e_{\ell_j(t)}}(\Phi(\widehat W_j(t)))-\widehat W_j(t).
\]
Both $\norm{W_j-\widehat W_j}$ and $e_{\ell_j(t)}$ tend to zero
uniformly on every compact time interval. Set
$\bar e:=\max\{1,\sup_k e_k\}<\infty$ and define on
$X:=K\times[0,\bar e]$ the map
\[
 \mathcal H(x,e):=\{(s-x,0):s\in\mathcal B_e(\Phi(x))\}.
\]
It has nonempty compact convex values and is upper semicontinuous:
the approximate minimization inequalities pass to limits in $(x,e,s)$,
and all values lie in the fixed compact set $(K-K)\times\{0\}$.
Apply Theorem~\ref{thm:appendix-di-compactness} to
$f_j=(W_j,0)$ and
$\widehat f_j=(\widehat W_j,e_{\ell_j})$.
The limit $(W,0)$ solves the exact differential inclusion, so the same
gap-decay contradiction gives
$\sup_{(x_k)\in\mathfrak X_e}\operatorname{dist}(x_k,\mathcal S)\to0$.
In the unique-solution polytopal case, identification is automatic
if $F_*=K$. Otherwise choose the common index large enough that also
$e_k<g_*/2$, with $g_*$ as in Lemma~\ref{lem:uniform-fw}.
A selected vertex outside $F_*$ would then
satisfy
\[
 g_*/2\leq\ip{\Phi(x_k)}{s_k-x^*}\leq e_k<g_*/2,
\]
which is impossible. This proves uniform identification for the
inexact trajectories.

Now take $e_k=D(k+1)^{-p}$. The decomposition, tangential shift, and
Taylor bound in the proof of Theorem~\ref{thm:app-fw} remain valid
after a common index $N$. If $U_*=\{0\}$, the rescaled state is
constant after identification and the claim follows immediately.
Otherwise, multiplying the inexact oracle inequality by $\rho_k$
adds $\rho_k\delta_k$ to \eqref{eq:boundary-oracle-error}.
Using the notation of that proof, put
\[
 E_m:=\max_{N\leq k<m}\rho_k\delta_k.
\]
The tolerance for updates with $N\leq k<m$ is at most
$2R_*\epsilon(Y_m+Q_N)+E_m$. Lemma~\ref{lem:absorb}, applied with
$b=2R_*\epsilon Q_N+E_m$ and $\eta=2R_*\epsilon$, gives
\[
 Y_m\leq
 \frac{C_{w,*}(W_m+Q_N+2R_*\epsilon Q_N+E_m)}
 {1-2C_{w,*}R_*\epsilon}.
\]
Here
\[
 E_m\leq D\max_{N\leq k<m}\frac{\rho_k}{(k+1)^p}
       =O(1+m^{(\alpha-p)_+}).
\]
Combining this with
$W_m=O(1+m^{(\alpha-1)_+})$ and $\rho_m\asymp m^\alpha$ proves
the asserted rate. The diameter bound covers the finite prefix,
and all constants are uniform over the stated family.
\end{proof}

\subsection{Gap rates with inexact oracles}\label{app:inexact-gaps}
\begin{proof}[Proof of the gap rates in Section~\ref{sec:inexact-extensions}]
Put $r=\min\{\alpha,1,p\}$.
The linear bound on $V$ in the proof of
Corollary~\ref{cor:gap-distance} still applies. Once the selected
vertices lie in $F_*$, the scalar
$b_k=\ip{\Phi(x^*)}{x_k-x^*}$ satisfies
$b_{k+1}=(1-\gamma_k)b_k$ exactly, even with an inexact oracle.
Thus $b_k=O(k^{-\alpha})$, and Lemma~\ref{lem:dual-growth} gives
$G(x_k)=O(k^{-\alpha}+k^{-2r})$.
\end{proof}

\subsection{Bounded delays and score-evaluation errors}\label{app:delays-proof}
\begin{proof}[Proof of Corollary~\ref{cor:delays}]
Let $L$ be a Lipschitz constant for $\Phi$ on $K$, let
$D_K:=\operatorname{diam}(K)$, and put $G:=\max\{1,\alpha\}$.
Since $\beta\geq\alpha$, we have $\gamma_i\leq G/(i+1)$, so
$\norm{x_{i+1}-x_i}\leq GD_K/(i+1)$. For $k\geq2\tau$ and $\tau\geq1$,
\[
 \norm{x_k-x_{j_k}}\leq GD_K\sum_{i=j_k}^{k-1}\frac1{i+1}
 \leq\frac{GD_K\tau}{k-\tau+1}\leq\frac{2GD_K\tau}{k+1}.
\]
For $k<2\tau$, the diameter bound gives the same final bound because
$G\geq1$; for $\tau=0$ the difference is zero. Thus
\[
 \norm{\widehat\Phi_k-\Phi(x_k)}\leq\frac{E_0+2LGD_K\tau}{k+1}.
\]
Comparing with an exact minimizer $v_k$ for $\Phi(x_k)$ gives
\[
 \ip{\Phi(x_k)}{s_k-v_k}
 \leq\frac{H_0}{k+1}
       +\ip{\Phi(x_k)-\widehat\Phi_k}{s_k-v_k}
 \leq\frac{H_0+D_KE_0+2LGD_K^2\tau}{k+1}.
\]
Theorem~\ref{thm:inexact} applies with $p=1$ and
$D=H_0+D_KE_0+2LGD_K^2\tau$.
\end{proof}

\end{document}

%% file: proof_roadmap.tex
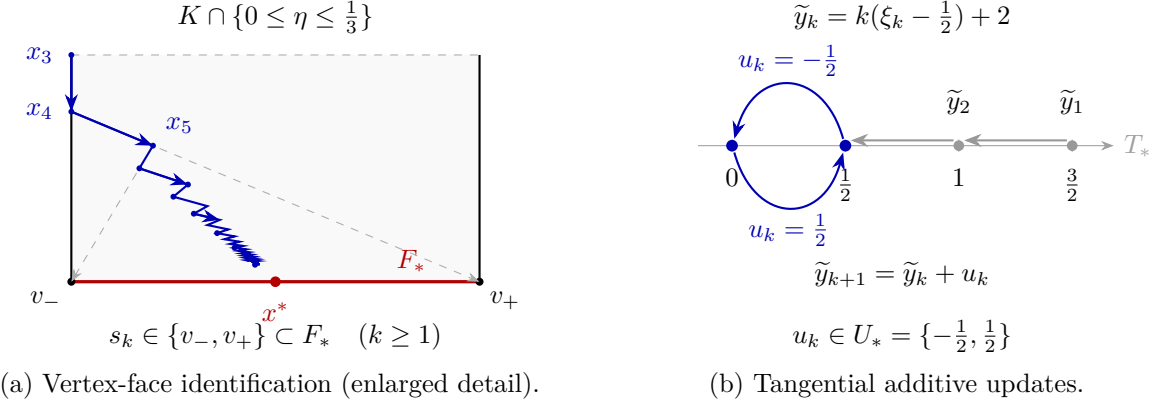
\begin{figure}[H]
\centering
\begin{subfigure}[t]{.49\textwidth}
\centering
\begin{tikzpicture}[x=1cm,y=1cm,>=Stealth,font=\small]
 \useasboundingbox (-.65,-.95) rectangle (5.95,3.8);
 \fill[gray!5] (0,0) rectangle (5.4,3);
 \draw[black,thick] (0,3)--(0,0)--(5.4,0)--(5.4,3);
 \draw[dashed,gray!65] (0,3)--(5.4,3);
 \node at (2.7,3.5) {$K\cap\{0\leq\eta\leq\tfrac13\}$};
 \draw[red!70!black,very thick] (0,0)--(5.4,0);
 \node[red!70!black,above] at (4.5,0) {$F_*$};
 \foreach \p in {(0,0),(5.4,0)} \fill \p circle (1.5pt);
 \node[below left] at (0,0) {$v_-$};
 \node[below right] at (5.4,0) {$v_+$};
 \fill[red!70!black] (2.7,0) circle (2pt);
 \node[red!70!black,below=4pt] at (2.7,0) {$x^*$};
 \draw[dashed,gray!65,->] (0,2.25)--(5.4,0);
 \draw[dashed,gray!65,->] (1.08,1.8)--(0,0);
 \draw[blue!70!black,thick] (0.000000,3.000000)--(0.000000,2.250000)--(1.080000,1.800000)--(0.900000,1.500000)--(1.542857,1.285714)--(1.350000,1.125000)--(1.800000,1.000000)--(1.620000,0.900000)--(1.963636,0.818182)--(1.800000,0.750000)--(2.076923,0.692308)--(1.928571,0.642857)--(2.160000,0.600000)--(2.025000,0.562500)--(2.223529,0.529412)--(2.100000,0.500000)--(2.273684,0.473684)--(2.160000,0.450000)--(2.314286,0.428571)--(2.209091,0.409091)--(2.347826,0.391304)--(2.250000,0.375000)--(2.376000,0.360000)--(2.284615,0.346154)--(2.400000,0.333333)--(2.314286,0.321429)--(2.420690,0.310345)--(2.340000,0.300000)--(2.438710,0.290323)--(2.362500,0.281250)--(2.454545,0.272727)--(2.382353,0.264706)--(2.468571,0.257143)--(2.400000,0.250000)--(2.481081,0.243243)--(2.415789,0.236842)--(2.492308,0.230769)--(2.430000,0.225000);
 \fill[blue!70!black] (0.000000,3.000000) circle (1.1pt);
 \fill[blue!70!black] (0.000000,2.250000) circle (1.1pt);
 \fill[blue!70!black] (1.080000,1.800000) circle (1.1pt);
 \fill[blue!70!black] (0.900000,1.500000) circle (1.1pt);
 \fill[blue!70!black] (1.542857,1.285714) circle (1.1pt);
 \fill[blue!70!black] (1.350000,1.125000) circle (1.1pt);
 \fill[blue!70!black] (1.620000,0.900000) circle (1.1pt);
 \fill[blue!70!black] (1.928571,0.642857) circle (1.1pt);
 \fill[blue!70!black] (2.160000,0.450000) circle (1.1pt);
 \fill[blue!70!black] (2.340000,0.300000) circle (1.1pt);
 \fill[blue!70!black] (2.430000,0.225000) circle (1.1pt);
 \draw[->,blue!70!black,thick] (0.000000,3.000000)--(0.000000,2.250000);
 \draw[->,blue!70!black,thick] (0.000000,2.250000)--(1.080000,1.800000);
 \draw[->,blue!70!black,thick] (0.900000,1.500000)--(1.542857,1.285714);
 \draw[->,blue!70!black,thick] (1.620000,0.900000)--(1.963636,0.818182);

 \node[blue!70!black,left=3pt] at (0,3) {$x_3$};
 \node[blue!70!black,left=3pt] at (0,2.25) {$x_4$};
 \node[blue!70!black,above right=1pt] at (1.08,1.8) {$x_5$};
 \node at (2.7,-.72) {$s_k\in\{v_-,v_+\}\subset F_*\quad(k\geq1)$};
\end{tikzpicture}
\caption{Vertex-face identification (enlarged detail).}
\label{subfig:face-entry}
\end{subfigure}\hfill
\begin{subfigure}[t]{.49\textwidth}
\centering
\begin{tikzpicture}[x=1cm,y=1cm,>=Stealth,font=\small]
 \useasboundingbox (-.65,-.95) rectangle (5.95,3.8);
 \node at (2.7,3.5) {$\widetilde y_k=k(\xi_k-\tfrac12)+2$};
 \draw[->,gray!75] (0,1.8)--(5.5,1.8) node[right] {$T_*$};
 \foreach \a/\lab in {.45/0,1.95/\frac12,3.45/1,4.95/\frac32}{
  \draw[gray!75] (\a,1.72)--(\a,1.88);
  \node[below=5pt] at (\a,1.8) {$\lab$};
 }
 \draw[->,gray!80,thick] (4.88,1.87)--(3.52,1.87);
 \draw[->,gray!80,thick] (3.38,1.87)--(2.02,1.87);
 \fill[gray!80] (4.95,1.8) circle (2pt);
 \fill[gray!80] (3.45,1.8) circle (2pt);
 \node[above=8pt] at (4.95,1.8) {$\widetilde y_1$};
 \node[above=8pt] at (3.45,1.8) {$\widetilde y_2$};
 \fill[blue!70!black] (.45,1.8) circle (2.2pt);
 \fill[blue!70!black] (1.95,1.8) circle (2.2pt);
 \draw[->,blue!70!black,thick] (1.92,1.89)
  .. controls (1.65,2.85) and (.75,2.85) .. (.48,1.89);
 \node[blue!70!black,above] at (1.2,2.58) {$u_k=-\tfrac12$};
 \draw[->,blue!70!black,thick] (.48,1.71)
  .. controls (.75,.74) and (1.65,.74) .. (1.92,1.71);
 \node[blue!70!black,below] at (1.2,1.03) {$u_k=\tfrac12$};
 \node at (2.7,.1) {$\widetilde y_{k+1}=\widetilde y_k+u_k$};
 \node at (2.7,-.72) {$u_k\in U_*=\{-\tfrac12,\tfrac12\}$};
\end{tikzpicture}
\caption{Tangential additive updates.}
\label{subfig:tangential-cycling}
\end{subfigure}
\caption{Face identification and cycling for an  example on
$K=[0,1]^2$, with $\Phi(\xi,\eta)=(\xi-\tfrac12+2\eta,
-2(\xi-\tfrac12)+\eta+\tfrac25)$, $\gamma_k=1/(k+1)$, and
$x_0=(1,0)$. The solution is $x^*=(\tfrac12,0)$ and
$F_*=[0,1]\times\{0\}$.
\subref{subfig:face-entry} Enlarged view of $0\leq\eta\leq1/3$, showing
$x_3,\ldots,x_{40}$; $x_1$ and $x_2$ lie above the displayed strip.
After $s_0=x_1=(0,1)$, every selected vertex lies in $F_*$;
the dashed segments point toward two selected vertices. The iterates
approach the face with $\eta_k=1/k>0$.
\subref{subfig:tangential-cycling} Rescaling makes the normal component
constant, and a tangential shift gives an additive update with increments $\pm1/2$.
Choosing $v_+$ at horizontal ties, $\widetilde y_k$ alternates between $0$
and $1/2$ for $k\geq4$. For nonlinear mappings,
Theorem~\ref{thm:weighted} controls the corresponding recursion with a
Taylor-remainder tolerance.}
\label{fig:proof-roadmap}
\end{figure}

%% file: face_identification.tex
\begin{figure}[!b]
\centering
\begin{tikzpicture}[x=1.35cm,y=1.12cm,>=Stealth,font=\small]
 \coordinate (A) at (0,0); \coordinate (B) at (4,0);
 \filldraw[fill=gray!9,draw=black,thick]
 (A)--(B)--(5.2,1.8)--(3.6,3.6)--(.8,3.2)--(-1,1.6)--cycle;
 \foreach \p in {(0,0),(4,0),(5.2,1.8),(3.6,3.6),(.8,3.2),(-1,1.6)}
   \fill \p circle (1.3pt);
 \node at (.6,2.4) {$K$};
 \draw[red!75!black,very thick] (A)--(B);
 \node[below left] at (A) {$v_1$};
 \node[below right] at (B) {$v_2$};
 \node[red!75!black,below=9pt] at (3.7,0) {$F_*$};
 \fill[red!75!black] (2,0) circle (1.7pt);
 \node[red!75!black,above left=1pt] at (2,0) {$x^*$};
 \draw[->,red!75!black,thick] (2,0)--(3.05,0)
   node[above=3pt] {$T_*$};
 \draw[->,gray!75!black,thick] (2,0)--(2,1.08)
   node[above left] {$T_*^\perp$};
 \draw[->,violet,thick] (2,0)--(2,-.83)
   node[below] {$-\Phi(x^*)$};
 \draw[blue!70!black,thick,->]
 (3.4,2.6)--(1.3,1.9)--(2.5,1.3)--(1.6,.9)--(2.25,.62)
 --(1.85,.42)--(2.12,.28)--(1.95,.18)--(2.06,.10);
 \node[blue!70!black,above right] at (3.4,2.6) {$x_k$};
 \draw[dashed,gray] (1.3,1.9)--(A);
 \draw[dashed,gray] (2.5,1.3)--(B);
 \node[anchor=west,align=left] at (5.45,1.6)
 {selected vertices\\$s_k\in\{v_1,v_2\}$\\for $k\geq k_0$};
\end{tikzpicture}
\caption{Exposed-face identification, shown schematically.
Selected vertices eventually lie in $F_*$, while the iterates need not
enter that face. The rescaled error $z_k=\rho_k(x_k-x^*)$ has a constant
normal component $q$ and moves only along $T_*$. Writing its tangential
component as $y_k$, the substitution $\widetilde y_k=y_k+A_*^{-1}c$
gives $A_*\widetilde y_k=A_*y_k+c$. The normal component $q$ stays unchanged.}
\label{fig:face-identification}
\end{figure}
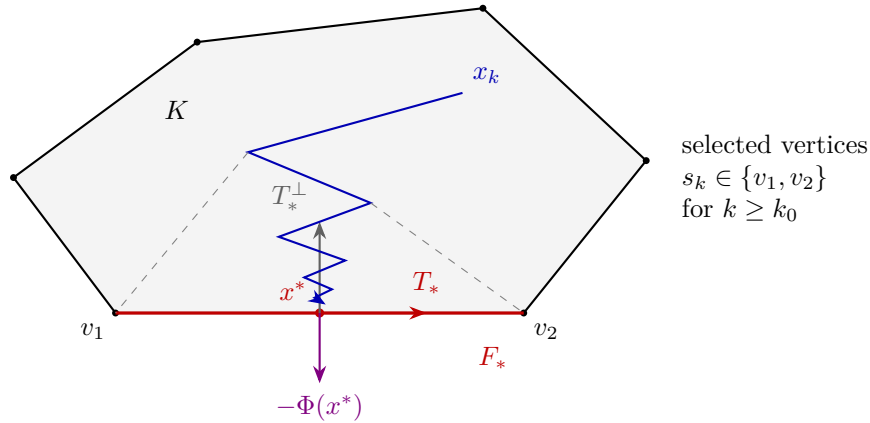

%% file: lower_bound.tex
\begin{figure}[!b]
\centering
\begin{tikzpicture}[x=4.0cm,y=3.8cm,>=Stealth,font=\small]
 \fill[gray!7] (0,0) rectangle (1,1);
 \draw[thick] (0,0) rectangle (1,1);
 \foreach \p in {(0,0),(1,0),(0,1),(1,1)} \fill \p circle (1.2pt);
 \node[below left] at (0,0) {$(0,0)$};
 \node[below right] at (1,0) {$(1,0)$};
 \node[above left] at (0,1) {$(0,1)$};
 \node[above right] at (1,1) {$(1,1)$};
 \fill[red!75!black] (.5,0) circle (1.7pt);
 \node[red!75!black,below=12pt] at (.5,0) {$x^*=(\tfrac12,0)$};
 \draw[blue!70!black,thick,->]
 (.84,.92)--(.46,.51)--(.65,.33)--(.43,.22)--(.59,.16)
 --(.44,.12)--(.56,.09)--(.46,.07)--(.54,.055)--(.48,.043);
 \node[blue!70!black,above left] at (.84,.92) {$x_k$};
 \draw[<->,blue!70!black] (.40,.62)--(.67,.62);
 \node[blue!70!black,anchor=east,align=right] at (-.09,.64)
 {oscillation\\$\Delta_k\geq\gamma_k/8$\\eventually};
 \draw[blue!70!black,->] (-.07,.62)--(.38,.62);
 \draw[decorate,decoration={brace,amplitude=4pt,mirror},gray]
 (1.06,0)--(1.06,.92);
 \node[anchor=west,align=left] at (1.17,.49)
 {memory\\$\displaystyle\eta_k=\eta_n\prod_{j=n}^{k-1}(1-\gamma_j)$\\[3pt]
 $\displaystyle\phantom{\eta_k}=\eta_n\rho_k^{-1}$};
\end{tikzpicture}
\caption{The two mechanisms in the proof of Theorem~\ref{thm:schedule-characterization},
shown schematically on $K=[0,1]^2$ with
$\Phi(\xi,\eta)=(\xi-1/2,\eta)$. The horizontal coordinate forces an
adjacent-iterate oscillation bound; the vertical coordinate retains
a fraction $\rho_k^{-1}$ of its positive value $\eta_n$ at the starting index.
The drawn path illustrates the geometry, rather than a specified schedule.}
\label{fig:lower-bound}
\end{figure}
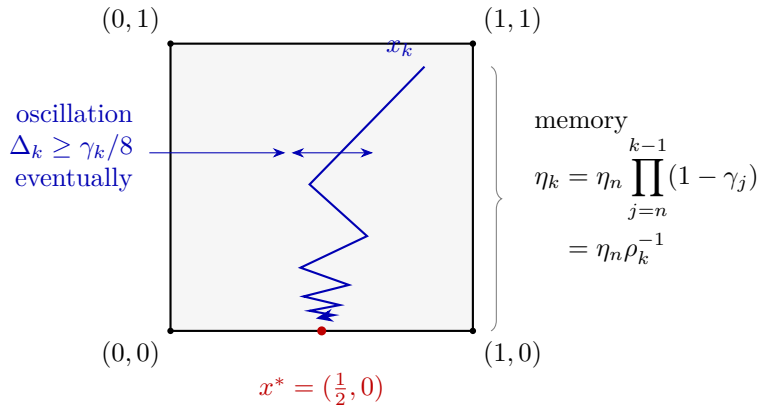

%% file: schedule_rates.tex
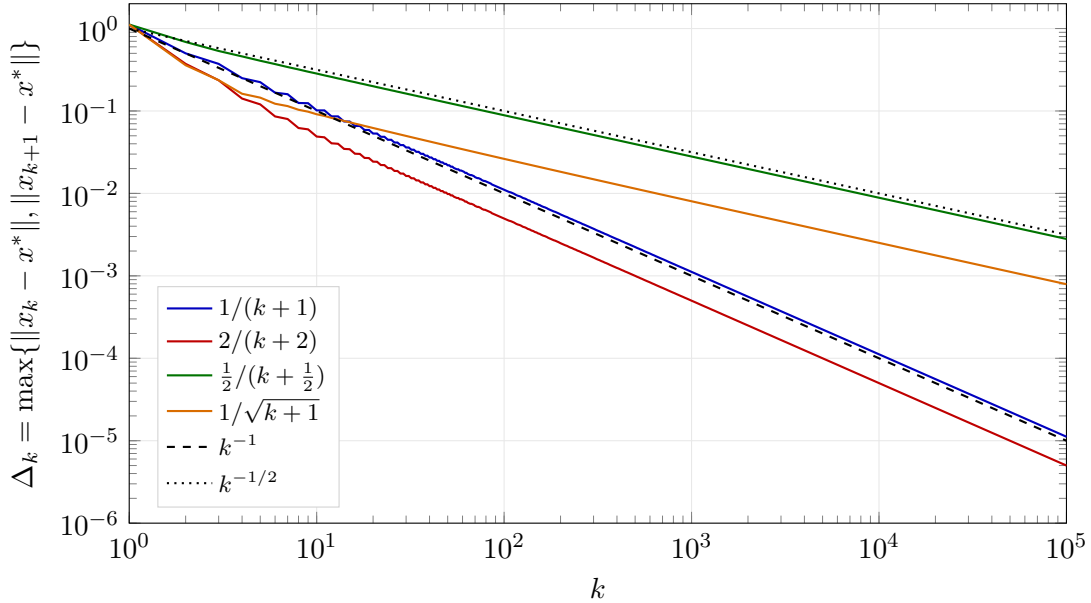
\begin{figure}[htbp]
\centering
\begin{tikzpicture}
\begin{loglogaxis}[
 width=.86\textwidth,height=.52\textwidth,
 xlabel={$k$},ylabel={$\Delta_k=\max\{\|x_k-x^*\|,\|x_{k+1}-x^*\|\}$},
 xmin=1,xmax=1e5,ymin=1e-6,ymax=2,
 legend style={font=\footnotesize,draw=gray!35,fill=white},
 legend pos=south west,legend cell align=left,
 grid=major,grid style={gray!18},
 every axis plot/.append style={thick,mark=none},
]
\addplot[blue!75!black] table[x=k,y=a1] {rates.dat};
\addplot[red!75!black] table[x=k,y=a2] {rates.dat};
\addplot[green!45!black] table[x=k,y=a05] {rates.dat};
\addplot[orange!85!black] table[x=k,y=sqrt] {rates.dat};
\addplot[black,dashed,domain=1:1e5,samples=2] {1/x};
\addplot[black,dotted,domain=1:1e5,samples=2] {1/sqrt(x)};
\legend{$1/(k+1)$,$2/(k+2)$,$\tfrac12/(k+\tfrac12)$,
        $1/\sqrt{k+1}$,$k^{-1}$,$k^{-1/2}$}
\end{loglogaxis}
\end{tikzpicture}
\caption{Numerical illustration on the fixed instance of
Theorem~\ref{thm:schedule-characterization}, with $x_0=x^*$ and the exact
tied selection $s_0=(1,1)$. Thereafter choose $a_k=1$ at horizontal ties
and $b_k=0$. We plot the two-iterate envelope from
Theorem~\ref{thm:schedule-characterization}, sampled at logarithmically
spaced indices through $10^5$. The first two schedules have order $k^{-1}$,
and the latter two have order $k^{-1/2}$. For harmonic steps,
$k\|x_k-x^*\|$ equals $1$ at even indices and $\sqrt5/2$ at odd indices;
the envelope displays the rate without this parity oscillation.}
\label{fig:schedule-rates}
\end{figure}